\documentclass[11pt, a4paper]{amsart}
\usepackage{amssymb,amsmath,amscd,url}
\usepackage{tikz-cd}
\usepackage{graphicx}

\usepackage[makeroom]{cancel}

\newtheorem{theorem}{Theorem}[section] 
\newtheorem{lemma}[theorem]{Lemma}     
\newtheorem{corollary}[theorem]{Corollary}
\newtheorem*{theoremn}{Corollary \ref{duality}}

\newtheorem*{corollaryn}{Corollary \ref{affine theorem}}
\newtheorem{proposition}[theorem]{Proposition}
\theoremstyle{remark}
\newtheorem{remark}[theorem]{Remark}
\theoremstyle{definition}
\newtheorem{definition}[theorem]{Definition}
\usepackage{amsfonts,url}
\newcommand{\cl}{\mathcal{C}\ell}

\newcommand{\la}{\langle}
\newcommand{\ra}{\rangle}
\newcommand{\ol}{\overline}
\newcommand{\e}{\mathbf{e}}
\newcommand{\bQ}{\mathbf{Q}}
\newcommand{\bF}{\mathbf{F}}
\newcommand{\eps}{\varepsilon}

\DeclareMathOperator{\coker}{coker}

\renewcommand{\simeq}{\cong}
\newcommand{\tensor}{\hspace{0.1em}\widehat{\otimes}\hspace{0.15em}}

\def\Tr{\mathrm{Tr}}
\def\fTr{\mathfrak{Tr}}

\newcommand{\Z}{\mathbb{Z}}
\newcommand{\C}{\mathbb{C}}

\def\GL{{\rm GL}}
\def\O{{\rm O}}

\def\cE{{\mathcal E}}
\def\cF{{\mathcal F}}

\def\Ext{{\rm Ext}}

\def\PSU{{\rm PSU}}

\def\SO{{\rm SO}}

\def\SU{{\rm SU}}
\def\GL{{\rm GL}}

\def\U{{\rm U}}

\def\Hom{{\rm Hom}}

\def\cP{{\mathcal P}}
\def\cQ{{\mathcal Q}}
\def\cS{{\mathcal S}}

\def\cU{{\mathcal U}}

\def\cZ{{\mathcal Z}}

\def\bH{{\mathbf H}}
\def\bL{{\mathbf L}}
\def\bS{{\mathbf S}}

\def\fb{{\mathfrak b}}
\def\fa{{\mathfrak a}}

\def\fx{{\mathfrak x}}
\def\fy{{\mathfrak y}}

\def\ft{{\mathfrak t}}

\def\fG{{\mathfrak G}}

\def\ft{{\mathfrak t}}

\def\Sp{{\rm Sp}}

\def\U{{\mathbf U}}

\def\id{\mathrm{id}}

\def\alphaCt{\mathbf{A}_{C_0(\ft)}}

\def\rtime{\!\rtimes\!}

\title[Extended affine Weyl groups]{Poincar\'e Duality and Langlands Duality for Extended Affine Weyl groups}

\author{Graham A. Niblo, Roger Plymen and Nick Wright}
\address{Mathematical Sciences, University of Southampton, SO17 1BJ,  England}
\email{g.a.niblo@soton.ac.uk, r.j.plymen@soton.ac.uk, wright@soton.ac.uk}
\thanks{The third author was partially supported by EPSRC grant EP/J015806/1.}

\begin{document}
\begin{abstract}
In this paper we construct an equivariant Poincar\'e duality between dual tori equipped with finite group actions. We use this to demonstrate that Langlands duality induces a rational isomorphism between the group $C^*$-algebras of extended affine Weyl groups at the level of $K$-theory.
\end{abstract}

\maketitle

\tableofcontents

\section*{Introduction}
Let $T$ be a compact torus and let $W$ be a finite group acting on $T$ with fixed point.  We construct a $W$-equivariant degree $0$ Poincar\'e duality between $C(T)$ and $C(T^\vee)$, where $T^\vee$ denotes the dual torus equipped with the dual action of $W$.

Moreover we show that there is a non-equivariant Poincar\'e duality between the crossed product algebras $C(T)\rtimes W$ and $C(T^\vee)\rtimes W$. Indeed we provide a general mechanism to descend equivariant Poincar\'e duality to Poincar\'e duality for crossed products. As far as we are aware this does not appear elsewhere in the literature. 

In the case when $W$ is trivial, our degree $0$ duality is connected to the Baum-Connes assembly map in the following way:  Let $T$ be a compact torus (equipped with the structure of a Lie group), and let $X^*(T), X_*(T)$ be the groups of characters and cocharacters respectively.  By definition the dual torus $T^\vee$ is the torus such that $X^*(T^\vee)=X_*(T)$ and $X_*(T^\vee)=X^*(T)$. Whence the Pontryagin dual of $X_*(T)$ is the torus $T^\vee$. The Baum-Connes assembly map for $X_*(T)$ gives a degree $0$ isomorphism
\[
K_*(T)\xrightarrow{\cong}K_*(C^*X_*(T))\cong K^*(T^\vee).
\]
This isomorphism agrees with our Poincar\'e duality, though this is not immediate from the definition of the two maps, see Section \ref{baumconnesconnection}.

For an isometric action of a group $W$ on a closed Riemannian manifold $M^n$, Kasparov's Poincar\'e duality \cite{K}, by contrast with our Poincar\'e duality, provides an isomorphism from $KK_W(C(M), \C)$ to $KK_W(\C,C_\tau(M))$ where $C_\tau(M)$ denotes the algebra of continuous sections of the Clifford bundle for the cotangent bundle $\tau$ of $M$.  See also Echterhoff \emph{et al}, \cite{EEK}. If the action is trivial and $M$ is a spin manifold, then the twisting by the Clifford algebra simply induces a dimension shift so Kasparov's Poincar\'e duality has degree $n$ modulo $2$. In the case where $M$ is a torus and $W$ is trivial, this is connected to our Poincar\'e duality via the Dirac-dual-Dirac method, which addresses the dimension shift. In the equivariant case the group acts nontrivially on the Clifford bundle, so the appearance of this bundle no longer simply gives a dimension shift. Indeed, for example, letting $\Z/2\Z$ act by complex conjugation on the $1$-dimensional torus $U(1)$, then $KK_W(\C,C_\tau(U(1)))$ is $\Z^3$ in dimension $0$ and $0$ in dimension $1$, which agrees with the \emph{unshifted} $K$-theory group $KK_W(\C,C(U(1)))$.

In this paper in order to describe the $KK$-cycles defining our Poincar\'e dualities explicitly, we have given direct proofs of the relevant properties of these cycles and their pairings. As remarked by the referee, it is in principle possible to obtain these elements by combining Kasparov's Poincar\'e duality elements with the Dirac and dual-Dirac cycles. Providing full details of this reduction to the known results is in itself somewhat complicated and we have opted to give the direct, self-contained argument.

\bigskip

As an application of our Poincar\'e duality we consider the case where $T$ is the maximal torus in a compact connected semi-simple Lie group and $W$ is the Weyl group. The dual torus is then the maximal torus in the Langlands dual Lie group. In general there is no $W$-equivariant homeomorphism between the two tori, hence \emph{a priori} one one would not expect them to have the same equivariant $K$-theory. However our Poincar\'e duality gives a canonical pairing between these two equivariant $K$-theory groups, and hence ignoring torsion these groups are isomorphic. Moreover our Poincar\'e duality also provides a canonical pairing between the $K$-theory of the extended affine Weyl groups of the original Lie group and its Langlands dual. This again yields an isomorphism up to torsion in $K$-theory, although these discrete groups are not typically isomorphic. In \cite{NPW} we explore this phenomenon in further detail and give detailed computations of these $K$-theory groups in a number of cases.

The connection between $T$-duality and Langlands duality has been studied by Daenzer-van Erp who showed that Langlands duality for complex reductive Lie groups can be implemented by $T$-dualization for groups whose simple factors are of type A, D or E, \cite{DvE}. This was generalised by Bunke-Nikolaus, see \cite{BN}. The study of $T$-duality in these papers, involves examining the Lie group viewed as a principal bundle of tori via the action of the maximal torus on the group. Here by contrast we study the Weyl group action on the maximal torus, instead of the maximal torus action on the Lie group. In both cases there is a natural duality arising from Langlands duality of root systems and the possible unification of these two perspectives would provide an interesting future project.

We would like to thank Maarten Solleveld for his helpful comments on the first version of this paper. We would also like to thank the editor Jonathan Rosenberg and the referee, whose careful reading of our paper led to the examination, in Section \ref{baumconnesconnection}, of the relationship between our Poincar\'e duality and the Baum-Connes conjecture.  We are also grateful to the referee for suggesting the simplified proof of Theorem \ref{descent of PD} presented here.

\section{Statement of Results}

Let $W$ be a finite group acting isometrically with a global fixed point on a flat Riemannian torus $T$, and let $\ft$ denote the universal cover of $T$.  The notation reflects the observation that $T$ can be equipped with the structure of an abelian Lie group with identity at the fixed point, and $\ft$ is then its Lie algebra which inherits a linear isometric action of $W$. Denote by $\Gamma$ the lattice in $\ft$ mapping to the identity in $T$, or equivalently the fundamental group of $T$. This inherits an action of $W$ from $\ft$.

Now let $T^\vee$ be the dual torus of $T$, that is, the group of characters of $\Gamma$. We similarly denote by $\ft^*$ the Lie algebra of $T^\vee$ (which is the dual space of $\ft$) and denote by $\Gamma^\vee$ the fundamental group of $T^\vee$. The action of $W$ on $T$ induces dual actions on $T^\vee,\ft^*$ and $\Gamma^\vee$.

Let $\cP\in KK_W(\C,C(T)\tensor C(T^\vee))$ denote the class of the Poincar\'e line bundle. To construct our Poincar\'e duality we will, in Section \ref{K-homology element}, define an element $\cQ\in KK_W(C(T^\vee)\tensor C(T), \C)$ given by a triple $(L^2(\ft)\tensor \cS,\rho,Q_0)$, for which $\cP,\cQ$ is a Poincar\'e duality pair. The operator is
\[
Q_0=\frac{\partial}{\partial y^j}\otimes \eps^j -2\pi i y^j\otimes \e_j
\]
where $\{\eps^j,\e_j:j=1,\dots, n\}$ denotes a suitable basis for $\ft^*\times\ft$ acting on a space of spinors $\cS$. The representation $\rho$ is defined by
\[
\rho(\sum_{\gamma\in\Gamma} a_\gamma e^{2\pi i\la \eta,\gamma\ra} \otimes f)\xi\otimes s = \sum_{\gamma\in\Gamma} a_\gamma \gamma\cdot(\tilde{f}\xi) \otimes s
\]
where $\gamma$ acts by translation on $L^2(\ft)$, $\tilde{f}$ denotes the lift of $f$ to a periodic function on $\ft$ and $\eta$ denotes a variable in $\ft^*$.

\newcommand{\PDTheorem}
{Let $T$ be a torus with flat Riemannian metric and $T^\vee$ its dual. Suppose that $W$ is a finite group acting isometrically on $T$ with a global fixed point. The elements $\cP,\cQ$ define a $W$-equivariant Poincar\'e duality in $KK$-theory from $C(T)$ to $C(T^\vee)$ and there is a `descended' non-equivariant Poincar\'e duality from $C_0(\ft)\rtimes (\Gamma\rtimes W)$ to $C_0(\ft^*)\rtimes (\Gamma^\vee\rtimes W)$.  This is natural in the sense that there is a commutative diagram
\[
\begin{CD}
KK^*_W(C(T),\C) @>\cong>> KK^*_W(\C,C(T^\vee)).\\
@VV\cong V @V\cong VV \\
KK^*(C_0(\ft)\rtimes (\Gamma\rtimes W),\C) @>\cong>> KK^*(\C,C_0(\ft^*)\rtimes(\Gamma^\vee\rtimes W))\\
\end{CD}
\]
where
\begin{itemize}
\item the top and bottom maps are induced by the Poincar\'e dualities,
\item the left-hand map is the composition of the $W$-equivariant Morita equivalence $C(T)\sim C_0(\ft)\rtimes \Gamma$ with the dual Green-Julg isomorphism in $K$-homology,
\item the right-hand map is its dual, i.e. the composition of the Morita equivalence $C(T^\vee)\sim C_0(\ft^*)\rtimes \Gamma^\vee$ with the Green-Julg isomorphism in $K$-theory.
\end{itemize}
}

\begin{theorem}\label{PD}
\PDTheorem
\end{theorem}

The vertical maps factor through $KK(C(T)\rtimes W,\C)$ on the left and $KK(\C,C(T^\vee)\rtimes W)$ on the right, and these may be identified (by Fourier-Pontryagin duality) with the groups $KK^*(C^*(\Gamma^\vee\rtimes W),\C)$ and $KK^*(\C,C^*(\Gamma\rtimes W))$ respectively.

\begin{theorem}\label{PD2}
Let $T$ be a torus with flat Riemannian metric and $T^\vee$ its dual. Suppose that $W$ is a finite group acting isometrically on $T$ with a global fixed point. The Poincar\'e duality from $C(T)$ to $C(T^\vee)$ descends to give a non-equivariant Poincar\'e duality as follows.
\[
\begin{CD}
KK^*_W(C(T),\C) @>\cong>> KK^*_W(\C,C(T^\vee)).\\
@VV\cong V @V\cong VV \\
KK^*(C^*(\Gamma^\vee\rtimes W),\C) @>\cong>> KK^*(\C,C^*(\Gamma\rtimes W))\\
\end{CD}
\]
where
\begin{itemize}
\item the top and bottom maps are induced by the Poincar\'e dualities,
\item the left-hand map is the composition of the $W$-equivariant Fourier-Pontryagin duality $C(T)\cong C^*(\Gamma^\vee)$ with the dual Green-Julg isomorphism in $K$-homology,
\item the right-hand map is its dual, i.e. the composition of the $W$-equivariant Fourier-Pontryagin duality $C(T^\vee)\cong C^*(\Gamma)$ with the Green-Julg isomorphism in $K$-theory.
\end{itemize}
\end{theorem}

In Section \ref{baumconnesconnection} we turn to the question of the relationship between the Baum-Connes assembly map and our Poincar\'e duality. In particular, we show that the following diagram of isomorphisms commutes. 
\[
\begin{CD}
KK_{\Gamma\rtimes W}^*(C_0(\ft),\C) @>\textrm{Baum-Connes}>> KK^*(\C,C^*(\Gamma\rtimes W))\\
@VV\textrm{dual Green-Julg} V @VV\textrm{Morita equivalence}V \\
KK^*(C_0(\ft)\rtimes (\Gamma\rtimes W),\C) @>\textrm{Poincar\'e duality}>> KK^*(\C,C_0(\ft^*)\rtimes(\Gamma^\vee\rtimes W))\\
\end{CD}
\]
Given the definitions of the maps this is, in some sense surprising since both the Baum-Connes and the dual Green-Julg maps factor through the descent map, which has target  $KK(\C,C_0(\ft^*)\rtimes(\Gamma^\vee\rtimes W) \otimes C^*(\Gamma\rtimes W))$. The corresponding square with this latter group in the top left corner as illustrated in Section \ref{baumconnesconnection}, does not commute. 

\bigskip
A case of particular interest is provided by the action of a Weyl group $W$ on a torus, provided by a root datum $(X^*, R, X_*, R^\vee)$. Let $W_a'=X_* \rtimes W$ be the corresponding extended affine Weyl group. The Langlands dual root system $(X_*, R^\vee, X^*, R)$ gives rise to a dual extended affine Weyl group $(W_a')^\vee=X^* \rtimes W$, which is not usually isomorphic to $W_a'$. However the Poincar\'e duality in Theorem \ref{PD2} provides an isomorphism between $K^*(C^*((W_a')^\vee))$ and $K_*(C^*(W_a'))$.

The Langlands duality between $W_a'$ and $(W_a')^\vee$ is further amplified by the following theorem.
\newcommand{\dualthm}{Let $G$ be a compact connected semisimple Lie group and $G^\vee$ its Langlands dual, with $W_a'$, $(W_a')^\vee$ the corresponding extended affine Weyl groups. Then  there is a rational isomorphism
$$K_*(C^*((W_a')^\vee))\cong K_*(C^*(W_a')).$$}
\begin{corollary}\label{duality}
\dualthm
\end{corollary}
\medskip

\newcommand{\affinecor}{Let $W_a'$ be the extended affine Weyl group of $G$, and let $W_a, W_a^\vee$ be the affine Weyl groups of $G$ and its Langlands dual $G^\vee$. If $G$ is of adjoint type then rationally
$$K_*(C^*(W_a^\vee))\cong K_*(C^*(W_a')).$$
If additionally $G$ is of type $A_n, D_n, E_6, E_7, E_8, F_4, G_2$ then rationally
$$K_*(C^*(W_a))\cong K_*(C^*(W_a')).$$}

In particular we obtain a duality between  affine and extended affine Weyl groups of the following form:

\begin{corollary}\label{affine theorem}
\affinecor
\end{corollary}

Recall that the extended affine Weyl group $W_a'$ is an extension of $W_a$ by the cyclic group $\pi_1(G)$ so the content of Corollary \ref{affine theorem} is that, surprisingly, this particular extension does not change the $K$-theory.

In a companion paper \cite{NPW} we explore this phenomenon in further detail and give detailed computations of these $K$-theory groups in a number of cases.

\section{Background}   \label{langlands section}

\subsection{Real Langlands duality}
Recall that a connected complex  reductive Lie group $\bH$ is classified by its root datum. That is a 4-tuple $(X^*, R, X_*, R^\vee)$ where $X^*$ is the lattice of characters on a maximal torus in $\bH$, and $X_*$ is the lattice of co-characters, or equivalently the fundamental group of the maximal torus. The set of roots $R\subset X^*$ is in bijection with the reflections in the Weyl group $W$ and in bijection with the set of coroots $R^\vee\subset X_*$. Root data classify connected complex  reductive Lie groups, in the sense that two such groups are isomorphic if and only if their root data are isomorphic (in the obvious sense).
The Langlands dual of $\bH$, denoted $\bH^\vee$ is then the unique connected complex  reductive Lie group associated to the dual root datum $(X_*, R^\vee, X^*, R)$. See \cite{B} for details.

One of the key motivations of this paper is that for extended affine Weyl groups the Baum-Connes correspondence should be thought of as an equivariant duality between maximal tori in a compact connected semisimple Lie group and its real Langlands dual. As in the complex case these are classified by their root data, and we can define the (real) Langlands dual by dualising the root datum as before. Since the real case is not as well known we recall the relationship with the complex case.

For a Lie group $G$, the \emph{complexification} $G_\C$ is a complex Lie group together with a morphism from $G$, satisfying the universal property that for any morphism of $G$ into a complex Lie group $\bL$ there is a unique factorisation through $G_\C$.

For $T$ a maximal torus in $G$, the complexification $\bS: = T_\C$ of $T$ is a maximal torus in $\bH: = G_\C$, and so the dual torus $\bS^\vee$ is well-defined in 
the dual group $\bH^\vee$.   Then $T^\vee$ is defined to be the maximal compact subgroup of $\bS^\vee$, and satisfies the condition
\[
(T^\vee)_\C =  \bS^\vee.
\]

The groups $X^*, X_*$ in the root datum are again the groups of characters and co-characters of the torus $T$ respectively. Dually $X_*, X^*$ are the groups of characters and co-characters on the dual torus $T^\vee$, giving the $T$-duality equation
\begin{align}\label{dualtorus2}
X^*(T^\vee) = X_*(T).
\end{align}
The torus $T^\vee$ is given explicitly by 
$T^\vee=\Hom(X_*(T), \U)$. 
 The \emph{Langlands dual} of $G$, denoted $G^\vee$, is defined to be a maximal compact subgroup of $\bH^\vee$ containing the torus $T^\vee$.

The process of passing to a maximal compact subgroup is inverse to complexification in the sense that complexifying $G^\vee$ recovers $\bH^\vee$.

\medskip
\subsubsection{A table of Langlands dual groups}  Given a compact connected semisimple Lie group $G$, the product $|\pi_1(G)| \cdot |\cZ(G)|$ is unchanged by Langlands duality, i.e.\ it agrees with the product $|\pi_1(G^\vee)| \cdot |\cZ(G^\vee)|$. This product is equal to the \emph{connection index}, denoted $f$, (see \cite[IX, p.320]{B}), which is defined in \cite[VI, p.240]{B}. The connection indices are listed in \cite[VI, Plates I--X,
p.265--292]{B}.

The following is a table of Langlands duals and connection indices for compact connected semisimple groups: 

\begin{center}
\begin{tabular}{l||l|l}
$G$ & $G^{\vee}$ & $f$\\
\hline
$A_n = \SU_{n + 1}$ & $\PSU_{n + 1}$  & $n+1$\\
$B_n = \SO_{2n+1}$ & $\Sp_{2n}$ & $2$\\
$C_n = \Sp_{2n}$ & $\SO_{2n+1}$ & $2$\\
$D_n = \SO_{2n}$ & $\SO_{2n}$ &$4$\\
$E_6$ & $E_6$ & $3$\\
$E_7$ & $E_7$ & $2$\\
$E_8$ & $E_8$ &$1$\\
$F_4$ & $F_4$ & $1$\\
$G_2$ & $G_2$ & $1$\\
\end{tabular}
\end{center}

In this table, 
the simply-connected form of $E_6$ (resp. $E_7$) dualises to the adjoint form of $E_6$ (resp. $E_7$).

The Lie group $G$ and its dual $G^\vee$ admit a common Weyl group 
\[
W = N(T)/T = N(T^\vee)/T^\vee.
\]
The $T$-duality equation (\ref{dualtorus2}) identifies the action of the Weyl group of $T$ on $X_*(T)$ with the dual action of the Weyl group of $T^\vee$ on $X^*(T^\vee)$.

\begin{remark}
In general,  $T$ and $T^\vee$ are \emph{not} isomorphic as $W$-spaces.  For example, let $G = \SU_3$ and take 
$T = \{(z_1, z_2, z_3) : z_j \in \U, z_1 z_2 z_3 = 1\}$. Then in homogeneous coordinates we have $T^\vee= \{(z_1: z_2: z_3) : z_j \in \U, z_1 z_2 z_3 = 1 \}$.  The Weyl group $W$ is the symmetric group $S_3$ which acts by permuting coordinates in both cases. Note that the torus $T$ admits three $W$-fixed points
whereas the unique $W$-fixed point in $T^\vee$ is the identity $(1:1:1) \in T^\vee$, hence $T$ and $T^\vee$ are not $W$-equivariantly isomorphic.  
\end{remark}

 The \emph{nodal group} of the torus $T$ is defined to be $\Gamma(T): = \ker (\exp: \ft \to T)$ and differentiating the action of the Weyl group $W$ we obtain a linear action of $W$ on the Lie algebra $\ft$ which restricts to an action on the nodal group $\Gamma(T)$. Indeed there is a $W$-equivariant isomorphism $X_*(T) \simeq \Gamma(T)$.

We remark that by definition $T^\vee$ is the Pontryagin dual of the nodal group $\Gamma(T)$. Moreover the natural action of $W$ on $T^\vee$ is the dual of the action on $\Gamma(T)$. Hence we have the following:

\begin{lemma}\label{gammadual}  Let $\widehat{\Gamma}$ denote the Pontryagin dual of $\Gamma=\Gamma(T)$.   Then we have a $W$-equivariant isomorphism 
\[
\widehat{\Gamma} \simeq T^\vee.
\]
and hence an isomorphism of $W$\!-\,$C^*$\!-algebras
\[
C^*(\Gamma) \simeq C(T^\vee).
\]
\end{lemma}

\subsection{Affine and Extended Affine Weyl groups}\label{affine Weyl section}

In this section we will give the  definitions of the affine and extended affine Weyl groups of a compact connected semisimple Lie group. As noted in the introduction these are semidirect products of lattices in the Lie algebra $\ft$ of a maximal torus $T$ by the Weyl group $W$. The affine Weyl group $W_a$ is a Coxeter group while the extended affine Weyl group contains $W_a$ as a finite index normal subgroup and the quotient  is the fundamental group of the Lie group $G$.

Let $p:\widetilde G\to G$ denote the universal cover and let $\widetilde T$ be the preimage of $T$ which is a maximal torus in $\widetilde G$.
We consider the following commutative diagram.
$$\begin{CD}
@.\Gamma(\widetilde T)@>>>\ft @>>>\widetilde T@>>> 0\\
@. @VV\iota V @VV{\id}V @VV{p|_{\widetilde T}}V\\
0@>>>\Gamma(T)@>>>\ft @>>>T@.
\end{CD}$$
By the snake lemma the sequence
$$
\begin{matrix}
\ker(\id)&\to& \ker(p|_{\widetilde T})&\to& \coker(\iota)&\to& \coker(\id)\\
||&&||&&||&&||\\
0& &\pi_1(G)& & \Gamma(T)/\Gamma(\widetilde T)& &0
\end{matrix}$$
is exact, hence $\Gamma(T)/\Gamma(\widetilde T)$ is isomorphic to $\pi_1(G)$. We thus have a map from $\Gamma(T)$ onto $\pi_1(G)$.  The kernel of this map (more commonly denoted $N(G,T)$)  is the nodal lattice $\Gamma(\widetilde T)$ for the torus $\widetilde T$ and we have:

\begin{definition}
The \emph{affine Weyl group of $G$} is 
\[
W_a(G) = \Gamma(\widetilde T) \rtimes W
\]
 and the \emph{extended
affine Weyl group of $G$} is 
\[
W'_a(G) = \Gamma(T) \rtimes W\]
where $W$ denotes the Weyl group of $G$.
\end{definition}

The following is now immediate.

\begin{lemma}\label{cover} Let $\widetilde{G}$ denote the universal cover of $G$ and let $\widetilde{T}$ denote
a maximal torus in $\widetilde{G}$.   Then we have
\[
W_a(G) = W'_a(\widetilde{G}) = W_a(\widetilde{G}).
\]
\end{lemma}

We remark that the extended affine Weyl group $W_a'(G)$ is a split extension of $W_a(G)$ by $\pi_1(G)$.

\section{Equivariant Poincar\'e Duality between $C(T)$ and $C(T^\vee)$}\label{poincare section}

We begin by recalling the general framework of Poincar\'e duality in $KK$-theory. For $\fG$-$C^*$-algebras $A,B$ a Poincar\'e duality is given by elements $\fa\in KK_\fG(B\tensor A,\C)$
and $\fb\in KK_\fG(\C,A\tensor B)$ with the property that
\begin{align}\label{PDisom}
\begin{split}
\fb\otimes_A\fa &= 1_B \in KK_\fG(B,B)\\
\fb\otimes_B\fa &= 1_A \in KK_\fG(A,A).
\end{split}
\end{align}

These then yield isomorphisms between the $K$-homology of $A$ and the $K$-theory of $B$ (and vice versa) given by
\begin{align*}
\fx \mapsto \fb\otimes_A\fx \in KK_\fG(\C,B) \text{ for } \fx \in KK(A,\C)\\
\fy \mapsto \fy\otimes_B\fa \in KK_\fG(A,\C) \text{ for } \fy \in KK(\C,B).
\end{align*}

Throughout this section $T$ will denote a torus with flat Riemannian metric, $T^\vee$ its dual torus and $W$ a finite group acting by isometries on $T$ (and dually on $T^\vee$). We will  construct elements  $\cQ\in KK_W(C(T^\vee)\tensor C(T),\C)$ and $\cP\in KK_W(\C,C(T)\tensor C(T^\vee))$ satisfying \ref{PDisom}.  

\subsection{The Poincar\'e line bundle }

Recall that the Poincar\'e line bundle over $T\times T^\vee$ is the bundle with total space given by the quotient of $\ft\times T^\vee\times \C$ under the action of $\Gamma$ defined by $\gamma(x,\chi,z)=(\gamma +x,\chi,\chi(\gamma)z)$. The projection onto the base $T\times T^\vee$ maps the $\Gamma$ orbit of $(x,\chi,z)$ to the $\Gamma$ orbit of $(x,\chi)$. Here we are identifying elements of $T^\vee$ with characters on $\Gamma$. We note that the bundle is $W$-equivariant with respect to the diagonal action of $W$ on $\ft\times T^\vee$, hence it defines an element in $W$-equivariant K-theory allowing it to play the role of the element $\cP$ in our  Poincar\'e duality. 

To place this in the language of $KK$-theory we consider sections of this bundle, which are given by functions $\sigma:\ft\times T^\vee\rightarrow \C$ such that $\sigma(\gamma+x, \chi)=\chi(\gamma)\sigma(x,\chi)$. They naturally form a module over $C(T\times T^\vee)$ and given two such sections we define $\langle \sigma_1, \sigma_2\rangle=\overline{\sigma_1}\sigma_2$. We note that this is a $\Gamma$ periodic function in the first variable, hence the inner product takes values in $C(T\times T^\vee)$, giving the space of sections the structure of a Hilbert module.

We will now give an alternative construction of this Hilbert module. Let $C_c(\ft)$ denote the space of continuous compactly supported functions on $\ft$ and equip this with a $C(T)\otimes C(T^\vee)$-valued inner product defined by
$$\la \phi_1,\phi_2\ra(x,\eta)=\sum_{\alpha,\beta \in\Gamma} \overline{\phi_1(x-\alpha)}\phi_2(x-\beta)e^{2\pi i \la \eta,\beta-\alpha\ra}.$$
We remark that the support condition ensures that this is a finite sum, and that it is easy to check that $\la \phi_1,\phi_2\ra(x,\eta)$ is $\Gamma$-periodic in $x$ and $\Gamma^\vee$-periodic in $\eta$.

The space $C_c(\ft)$ has a $C(T)\otimes \C[\Gamma]$-module structure
$$(\phi\cdot (f\otimes [\gamma])) = \phi(x+\gamma) \tilde f(x)$$
where we view the function $f$ in $C(T)$ as a $\Gamma$-periodic function $\tilde{f}$ on $\ft$.
 
Completing $C_c(\ft)$ with respect to the norm arising from the inner product, the module structure extends by continuity to give $\ol{C_c(\ft)}$ the structure of a $C(T)\tensor C^*(\Gamma)\cong C(T)\tensor C(T^\vee)$ Hilbert module. We denote this Hilbert module by $\cE$ and give this the trivial grading.

The group $W$ acts on $\ft$ and hence on $C_c(\ft)$ by $(w\cdot \phi)(x)=\phi(w^{-1}x)$. We have
$$(w\cdot(\phi\cdot (f\otimes [\gamma])))(x) = \phi(w^{-1}x+\gamma) \tilde f(w^{-1}x)=((w\cdot \phi)\cdot(w\cdot f\otimes[w\gamma]))(x)$$
so the action is compatible with the module structure.  Now for the inner product we have

\begin{align*}
\la w\cdot\phi_1,w\cdot\phi_2\ra(x,\eta)&=\sum_{\alpha,\beta \in\Gamma} \overline{(w\cdot\phi_1)(x-\alpha)}(w\cdot\phi_2)(x-\beta)e^{2\pi i \la \eta,\beta-\alpha\ra}\\
&=\sum_{\alpha,\beta \in\Gamma} \overline{\phi_1(w^{-1}x-w^{-1}\alpha)}\phi_2(w^{-1}x-w^{-1}\beta)e^{2\pi i \la \eta,\beta-\alpha\ra}\\
&=\sum_{\alpha',\beta' \in\Gamma} \overline{\phi_1(w^{-1}x-\alpha')}\phi_2(w^{-1}x-\beta')e^{2\pi i \la \eta,w(\beta'-\alpha')\ra}\\
&=\sum_{\alpha',\beta' \in\Gamma} \overline{\phi_1(w^{-1}x-\alpha')}\phi_2(w^{-1}x-\beta')e^{2\pi i \la w^{-1}\eta,\beta'-\alpha'\ra}\\
&=(w\cdot\la \phi_1,\phi_2\ra)(x,\eta).
\end{align*}
Hence $\cE$ is a $W$-equivariant Hilbert module.

The identification of the module $\cE$ with the sections of the Poincar\'e line bundle is given by the following analog of the Fourier transform. For each element $\phi\in C_c(\ft)$ set 

\[
\sigma(x,\chi)=\sum\limits_{\gamma\in \Gamma} \phi(x-\gamma)\chi(\gamma).
\]

It is routine to verify that $\sigma(x+\delta,\chi)=\chi(\delta)\sigma(x,\chi)$ hence $\sigma$ is a section of the Poincar\'e line bundle, and that the $W$ action on $C_c(\ft)$ corresponds precisely to the $W$ action on the bundle.

\begin{theorem}
The triple $(\cE,1,0)$, where $1$ denotes the identity representation of $\C$ on $\cE$, is a $W$-equivariant Kasparov triple defining an element $\cP$ in $KK_W(\C,C(T)\tensor C(T^\vee))$.
\end{theorem}

We remark that there is a connection with Fourier-Mukai duality. We recall that Fourier-Mukai duality is given by the map 

\[
\fx \mapsto p_{2*}(p_1^*\fx\otimes \cP),
\]
\noindent where $p_1, p_2$ are the projections of $T\times T^\vee$ onto the first and second factors. From the point of view of $K$-theory the subtlety is to interpret the wrong-way map $p_{2*}$. This should give a map from the  $W$-equivariant $K$-theory of $T\times T^\vee$ to the $W$ equivariant $K$-theory of $T^\vee$, but to make this well defined we must twist by the Clifford algebra $\cl(\ft)$. Specifically we can take 
\[
p_{2*}:= [D]\otimes 1_{C(T^\vee)}\in KK_W(C(T\times T^\vee)\otimes \cl(\ft), C(T^\vee)),
\]
where $[D]$ is the Dirac class in $KK_W(C(T)\otimes \cl(\ft), \C)$. The Fourier-Mukai map is then given by taking the Kasparov product with the element $p_1^*\cP i^* p_{2*}=\cP p_1^* i^* p_{2*}$ where $i$ is the diagonal inclusion of $T\times T^\vee$ into $(T\times T^\vee)^2$. We note that $p_1^* i^* p_{2*}$ is the tensor product of Kasparov's Poincar\'e duality element for $T$ (given by its Dirac element) with the identity on $C(T^\vee)$.

\subsection{Construction of the element $\cQ$ in  $KK_W(C(T^\vee)\tensor C(T),\C)$}\label{K-homology element}

We consider the differential  operator $Q_0$ on $\ft$ with coefficients in the Clifford algebra $\cl(\ft\times \ft^*)$ defined using Einstein summation convention by
$$Q_0=\frac{\partial}{\partial y^j}\otimes \eps^j -2\pi i y^j\otimes \e_j.$$
Here $\{\e_j=\frac{\partial}{\partial y^j}\}$ is an orthonormal basis for $\ft$ ,  $\{\eps^j\}$ denotes the dual basis of $\ft^*$ and we regard these as generators of the Clifford algebra $\cl(\ft\times \ft^*)$.

We view $Q_0$ as an unbounded operator on the Hilbert space $L^2(\ft)\tensor \cS$, where $\cS$ denotes the space of spinors
$\cS=\cl(\ft\times \ft^*)P$ with $P$ the projection $\prod_{j}\frac 12(1-i\e_j\eps^j)$ in the Clifford algebra. (The space $\cS$ is naturally equipped with a representation of $\cl(\ft\times \ft^*)$ by left multiplication inducing the action of $Q_0$.)

The subtlety in constructing an element in equivariant $KK$-theory is the need to ensure that $P$ is $W$-invariant with respect to the diagonal action of $W$ on $\ft\times \ft^*$ 
and hence that the action of $W$ on $\cl(\ft\times \ft^*)$ restricts to a representation on $\cS$. The corner algebra $P\cl(\ft\times \ft^*)P$ is $\C P$, which we identify  with $\C$, and this gives $\cS$ a canonical inner product given by $\la aP,bP\ra=Pa^*bP$.

As a simple example consider the $1$-dimensional case. Here $\cl(\ft\times \ft^*)=M_2(\C)$ and the two generators are $\e_1 = \begin{pmatrix}0&i\\i&0\end{pmatrix}$ and $\eps^1= \begin{pmatrix}0&-1\\1&0\end{pmatrix}$. The projection $P$ is therefore $ \begin{pmatrix}1&0\\0&0\end{pmatrix}$ so $\cS$ is the space of matrices of the form $ \begin{pmatrix}*&0\\{*}&0\end{pmatrix}$ and the operator is
\[
Q_0=\begin{pmatrix}0&-\frac{\partial}{\partial y^1}+2\pi y^1\\\frac{\partial}{\partial y^1}+2\pi y^1&0\end{pmatrix}.
\]
The off-diagonal elements are of course the $1$-dimensional annihilation and creation operators.

Returning to the general case we must now construct a representation of $C(T^\vee)\tensor C(T)$ on $L^2(\ft)\tensor \cS$.  It suffices to define commuting representations of $C(T^\vee)\tensor 1$ and $1\tensor C(T)$. The representation of $C(T)$ is the usual pointwise multiplication on $L^2(\ft)$ viewing elements of $C(T)$ as $\Gamma$-periodic functions on $\ft$. The representation of $C(T^\vee)$ involves the action of $\Gamma$ on $\ft$. 

For $a$ an affine isometry of $\ft$, let $L_a$ be the operator on $L^2(\ft)$ induced by the action of $a$ on $\ft$:
$$(L_a\xi)(y)=\xi(a^{-1}\cdot y).$$
For the function $\eta\mapsto e^{2\pi i\la \eta,\gamma\ra}$ in $C(T^\vee)$ we define
$$\rho(e^{2\pi i\la \eta,\gamma\ra})=L_\gamma\otimes 1_\cS.$$
Identifying $C(T^\vee)$ with $C^*(\Gamma)$ and identifying $L^2(\ft)$ with $\ell^2(\Gamma)\otimes L^2(X)$ where $X$ is a fundamental domain for the action of $\Gamma$, the representation of the algebra is given by the left regular representation on $\ell^2(\Gamma)$.

Consider the commutators of $Q_0$ with the representation $\rho$.  For $f\in C(T)$, the operator $\rho(f)$ commutes exactly with the second term $2\pi i y^j\otimes \e_j$ in $Q_0$, while, for $f$ smooth, the commutator of $\rho(f)$ with $\frac{\partial }{\partial y^j}\otimes \eps^j$ is given by the bounded operator $\frac{\partial f}{\partial y^j}\otimes \eps^j$. Now for the function $\eta\mapsto e^{2\pi i\la \eta,\gamma\ra}$ in $C(T^\vee)$ we have $\rho(e^{2\pi i\la \eta,\gamma\ra})=L_\gamma\otimes 1_\cS$.  This commutes exactly with the differential term of the operator, while
$$L_\gamma(2\pi i y^j)L_\gamma^*=2\pi i (y^j-\gamma^j)$$
hence the commutator $[L_\gamma\otimes 1_\cS,2\pi i y^j\otimes \e_j]$ is again bounded.

We have verified that $Q_0$ commutes with the representation $\rho$ modulo bounded operators, on a dense subalgebra of $C(T^\vee)\tensor C(T)$.  Thus to show that the triple
$$(L^2(\ft)\tensor \cS,\rho,Q_0)$$
is an unbounded Kasparov triple it remains to prove the following.

\begin{theorem}\label{ladder}
The operator $Q_0$ has compact resolvent.  It has $1$-dimensional kernel with even grading.
\end{theorem}

\begin{proof}
In the following argument we will \emph{not} use summation convention. We consider the following operators on $L^2(\ft)\tensor \cS$:
\begin{align*}p_j&=\frac{\partial}{\partial y^j}\otimes \eps^j\\
x_j&=-2\pi iy^j \otimes \e_j\\
q_j&=\frac{1}{2}(1+1\otimes i\e_j\eps^j)\\
A_j&=\frac{1}{2\sqrt{\pi}}(p_j+x_j)
\end{align*}
Since $A_j$ anti-commutes with $1\otimes i\e_j\eps^j$ we have $q_jA_j=A_j(1-q_j)$, hence we can think of $A_j$ as an off-diagonal matrix with respect to $q_j$. We write $A_j$ as $a_j+a_j^*$ where $a_j=q_jA_j=A_j(1-q_j)$ and hence $a_j^*=A_jq_j=(1-q_j)A_j$. We think of $a_j^*$ and $a_j$ as creation and annihilation operators respectively and we define a number operator $N_j=a_j^*a_j$. The involution $i\eps^j$ intertwines $q_j$ with $1-q_j$. We define $A_j',N_j'$ to be the conjugates of $A_j,N_j$ respectively by $i\eps^j$.  Note that
$$A_j'=\frac{1}{2\sqrt{\pi}}(p_j-x_j)$$
and hence
$$A_j^2=(A_j')^2+2\frac{1}{4\pi}[x_j,p_j]=(A_j')^2+1\otimes i\e_j\eps^j.$$
We have $N_j'=A_j'(1-q_j)A_j'=q_j(A_j')^2$. Thus
$$a_ja_j^*=q_jA_j^2q_j=q_jA_j^2=q_j(A_j')^2+q_j(1\otimes i\e_j\eps^j)=N_j'+q_j.$$
Hence the spectrum of $a_ja_j^*$ (viewed as an operator on the range of $q_j$) is the spectrum of $N_j'$ shifted by $1$. However $N_j'$ is conjugate to $N_j=a_j^*a_j$ so we conclude that
$$\Sp (a_ja_j^*)=\Sp (a_j^*a_j)+1.$$
But $\Sp (a_ja_j^*)\setminus\{0\}=\Sp (a_j^*a_j)\setminus\{0\}$ so we conclude that the spectrum is $\Sp (a_j^*a_j)=\{0,1,2,\dots\}$ while $\Sp (a_ja_j^*)=\{1,2,\dots\}$.

Now since the operators $A_j$ pairwise gradedly commute we have
$$Q_0^2=4\pi\sum_{j} A_j^2=4\pi\sum_{j} a_j^*a_j+a_ja_j^*$$
and noting that the summands commute we see that $Q_0^2$ has discrete spectrum.  To show that $(1+Q_0^2)^{-1}$ is compact, it remains to verify that $\ker Q_0$ is finite dimensional (and hence that all eigenspaces are finite dimensional). We have
$$\ker Q_0 = \ker Q_0^2 = \bigcap_j\ker  A_j^2=\bigcap_j \ker A_j.$$

Multiplying the differential equation $(p_j+x_j)f=0$ by $-\exp(\pi (y^j)^2\otimes i\eps^j\e_j)\eps^j$ we see that the kernel of $A_j$ is the space of solutions of the differential equation
$$\frac{\partial }{\partial y^j}(\exp(\pi (y^j)^2\otimes i\eps^j\e_j)f)=0$$whence for $f$ in the kernel we have
$$f(y^1,\dots,y^n)=\exp(-\pi (y^j)^2\otimes i\eps^j\e_j)f(y^1,\dots,y^{j-1},0,y^{j+1},\dots,y^n).$$
Since the solutions must be square integrable the values of $f$ must lie in the $+1$ eigenspace of the involution $i\eps^j\e_j$, that is, the range of the projection $1-q_j$. On this subspace the operator $\exp(-\pi (y^j)^2\otimes i\eps^j\e_j)$ reduces to $e^{-\pi (y^j)^2}(1-q_j)$.  Since the kernel of $Q_0$ is the intersection of the kernels of the operators $A_j$ an element of the kernel must have the form
$$f(y)=e^{-\pi|y|^2}\prod_j(1-q_j)f(0)$$
so the kernel is $1$-dimensional. Indeed the product $\prod_j(1-q_j)$ is the projection $P$ used to define the space of spinors $\cS=\cl(\ft\times\ft^*)P$, and hence $\prod_j(1-q_j)f(0)$ lies in the $1$-dimensional space $P\cS=P\cl(\ft\times\ft^*)P$ which has even grading.
\end{proof}

We have shown that $(L^2(\ft)\tensor \cS,\rho,Q_0)$ defines an unbounded Kasparov triple. It remains to establish  $W$-equivariance.

Let $V$ be a finite dimensional vector space and equip $V\otimes V^*$ with the natural diagonal action of $\GL(V)$. If $V$ is equipped with a non-degenerate symmetric bilinear form $g$ then we can form the Clifford algebra $\cl(V)$ and dually $\cl(V^*)$. The subgroup $\O(g)$ of $\GL(V)$, consisting of those elements preserving $g$, acts diagonally on $\cl(V)\tensor \cl(V^*)$ which we identify with $\cl(V\times V^*)$.

We say that an element $a$ of $\cl(V\times V^*)$ is \emph{symmetric} if there exists a $g$-orthonormal\footnote{We say that $\{\e_j\}$ is $g$-orthonormal if $g_{jk}=\pm\delta_{jk}$ for each $j,k$.} basis $\{\e_j: j=1,\dots,n\}$ with dual basis $\{\eps^j: j=1,\dots,n\}$ such that $a$ can be written as $p(\e_1\eps^1,\dots,\e_n\eps^n)$ where $p(x_1,\dots,x_n)$ is a symmetric polynomial.

\begin{proposition}\label{invariance of symmetric elements}
For any basis $\{\e_j\}$ of $V$ with dual basis $\{\eps^j\}$ for $V^*$, the Einstein sum $\e_j\otimes \eps^j$ in $V\otimes V^*$ is $\GL(V)$-invariant.

Suppose moreover that $V$ is equipped with a non-degenerate symmetric bilinear form $g$ and that the underlying field has characteristic zero. Then every symmetric element of $\cl(V)\tensor \cl(V^*)\cong \cl(V\times V^*)$ is $\O(g)$-invariant.
\end{proposition}
\begin{proof}
Identifying $V\otimes V^*$ with endomorphisms of $V$ in the natural way, the action of $\GL(V)$ is the action by conjugation and $\e_j\otimes \eps^j$ is the identity which is invariant under conjugation.

For the second part, over a field of characteristic zero the symmetric polynomials are generated by power sum symmetric polynomials $p(x_1,\dots,x_n)=x_1^k+\dots+x_n^k$, so it suffices to consider
\begin{align*}
p(\e_1\eps^1,\dots,\e_n\eps^n)&=(\e_1\eps^1)^k+\dots+(\e_n\eps^n)^k\\
&=(-1)^{k(k-1)/2}\Bigl((\e_1)^k(\eps^1)^k+\dots+(\e_n)^k(\eps^n)^k\Bigr).
\end{align*}
When $k$ is even, writing $(\e_j)^k=(\e_j^2)^{k/2}=(g_{jj})^{k/2}$ and similarly $(\eps^j)^k=(g^{jj})^{k/2}$, we see that each term $(\e_j)^k(\eps^j)^k$ is $1$ since $g_{jj}=g^{jj}=\pm1$ for an orthonormal basis. Thus $p(\e_1\eps^1,\dots,\e_n\eps^n)=n(-1)^{k(k-1)/2}$ which is invariant.

Similarly when $k$ is odd we get $(\e_j)^k(\eps^j)^k=\e_j\eps^j$ so
$$p(\e_1\eps^1,\dots,\e_n\eps^n)=(-1)^{k(k-1)/2}(\e_1\eps^1+\dots+\e_n\eps^n).$$
As the sum $\e_j\otimes \eps^j$ in $V\otimes V^*$ is invariant under $\GL(V)$, it is in particular invariant under $\O(g)$, and hence the sum $\e_j\eps^j$ is $\O(g)$-invariant in the Clifford algebra.
\end{proof}

\bigskip

Returning to our construction, the projection $P$ is a symmetric element of the Clifford algebra and hence is $W$-invariant by Proposition \ref{invariance of symmetric elements}. It follows that $\cS$ carries a representation of $W$. The space $L^2(\ft)$ also carries a representation of $W$ given by the action of $W$ on $\ft$ and we equip $L^2(\ft)\tensor \cS$ with the diagonal action of $W$.

To verify that the representation $\rho$ is $W$-equivariant it suffices to consider the representations of $C(T)$ and $C(T^\vee)$ separately. As the exponential map $\ft\to T$ is $W$-equivariant it is clear that the representation of $C(T)$ on $L^2(\ft)$ by pointwise multiplication is $W$-equivariant.

For $e^{2\pi i\la \eta,\gamma\ra}\in C(T^\vee)$ we have $w\cdot(e^{2\pi i\la \eta,\gamma\ra})=e^{2\pi i\la w^{-1}\cdot\eta,\gamma\ra}=e^{2\pi i\la \eta,w\cdot\gamma\ra}$ thus $\rho(w\cdot(e^{2\pi i\la \eta,\gamma\ra}))=L_{w\cdot\gamma}\otimes 1_\cS=L_{w}L_{\gamma}L_{w^{-1}}\otimes 1_\cS$. Thus the representation of $C(T^\vee)$ is also $W$-equivariant.

It remains to check that the operator $Q_0$ is $W$-equivariant. By definition
$$Q_0=\frac{\partial}{\partial y^j}\otimes \eps^j-2\pi i y^j\otimes \e_j.$$
Now by Proposition \ref{invariance of symmetric elements} $\frac{\partial}{\partial y^j}\otimes \eps^j=\e_j\otimes \eps^j$ is a $\GL(\ft)$-invariant element of $\ft\otimes \ft^*$ and so in particular it is $W$-invariant. Writing $y^j=\la \eps^j, y\ra$ the $W$-invariance of the second term again follows from invariance of $\e_j\otimes \eps^j$.

Hence we conclude the following.

\begin{theorem}
The triple $(L^2(\ft)\tensor \cS,\rho,Q_0)$ constructed above defines an element $\cQ$ of $KK_W(C(T^\vee)\tensor C(T),\C)$.
\end{theorem}

\subsection{The Kasparov product $\cP\otimes_{C(T^\vee)}\cQ$} We will compute the Kasparov product of the Poincar\'e line bundle $\cP\in KK_W(\C,C(T)\tensor C(T^\vee))$ with our inverse $\cQ\in KK_W(C(T^\vee)\tensor C(T),\C)$ where the product is taken over $C(T^\vee)$ (not $C(T)\tensor C(T^\vee)$).

Recall that $\cP$ is given by the Kasparov triple $(\cE,1,0)$ where $\cE$ is the completion of $C_c(\ft)$ with the inner product
$$\la \phi_1,\phi_2\ra(x,\eta)=\sum_{\alpha,\beta \in\Gamma} \overline{\phi_1(x-\alpha)}\phi_2(x-\beta)e^{2\pi i \la \eta,\beta-\alpha\ra}$$
in $C(T)\tensor C(T^\vee)$. As above $\cQ$ is given by the triple $(L^2(\ft)\tensor \cS,\rho,Q_0)$.

To form the Kasparov product we must take that tensor product of $\cE$ with $L^2(\ft)\tensor \cS$ over $C(T^\vee)$ and as the operator in the first triple is zero, the operator required for the Kasparov product can be any connection for $Q_0$.

We note that the representation $\rho$ is the identity on $\cS$ and hence
$$\cE\tensor_{C(T^\vee)}(L^2(\ft)\tensor \cS)=(\cE\tensor_{C(T^\vee)}L^2(\ft))\tensor \cS.$$
Thus we can focus on identifying the tensor product $\cE\tensor_{C(T^\vee)}L^2(\ft)$. By abuse of notation we will also let $\rho$ denote the representation of $C(T)\tensor C(T^\vee)$ on $L^2(\ft)$.

As we are taking the tensor product over $C(T^\vee)$, not over $C(T)\tensor C(T^\vee)$, we are forming the Hilbert module
$$(\cE\tensor C(T))\tensor_{C(T)\tensor C(T^\vee)\tensor C(T)}(C(T)\tensor L^2(\ft))$$
however since the algebra $C(T)$ is unital, it suffices to consider elementary tensors of the form $(\phi\otimes 1)\otimes (1\otimes \xi)$. Where there is no risk of confusion we will abbreviate these are $\phi\otimes \xi$

Let $\phi_1,\phi_2\in C_c(\ft)$ and let $\xi_1,\xi_2$ be elements of $L^2(\ft)$.  Then
\begin{align*}
\la \phi_1\otimes\xi_1,\phi_2\otimes\xi_2\ra &=\la 1\otimes \xi_1,(1\otimes\rho)(\la \phi_1,\phi_2\ra\otimes 1)(1\otimes \xi_2)\ra.
\end{align*}
The operator $(1\otimes\rho)(\la \phi_1,\phi_2\ra\otimes 1)$ corresponds to a field of operators
\begin{align*}(1\otimes\rho)(\la \phi_1,\phi_2\ra\otimes 1)(x)&=\sum_{\alpha,\beta \in\Gamma} \overline{\phi_1(x-\alpha)}\phi_2(x-\beta)\otimes \rho(e^{2\pi i \la \eta,\beta-\alpha\ra}\otimes 1)\\
&=\sum_{\alpha,\beta \in\Gamma} \overline{\phi_1(x-\alpha)}\phi_2(x-\beta)\otimes L_\alpha^*L_\beta
\end{align*}
and so
\begin{align*}
\la \phi_1\otimes\xi_1,\phi_2\otimes\xi_2\ra(x) &=\sum_{\alpha,\beta \in\Gamma} \overline{\phi_1(x-\alpha)}\phi_2(x-\beta)\la L_\alpha\xi_1, L_\beta\xi_2\ra\\
&=\la \sum_{\alpha\in \Gamma}\phi_1(x-\alpha)L_\alpha\xi_1,\sum_{\beta\in \Gamma}\phi_2(x-\beta) L_\beta\xi_2\ra.
\end{align*}
We note that $x\mapsto \sum_{\alpha\in \Gamma}\phi_1(x-\alpha)L_\alpha\xi_1$ is a continuous $\Gamma$-equivariant (and hence bounded) function from $\ft$ to $L^2(\ft)$. Let $C(\ft,L^2(\ft))^\Gamma$ denote the space of such functions equipped with the $C(T)$ module structure of pointwise multiplication in the first variable and gives the pointwise inner product $\la g_1,g_2\ra(x)=\la g_1(x),g_2(x)\ra$. We remark that equivariance implies this inner product is a $\Gamma$-periodic function on $\ft$.

The above calculation show that $\cE\tensor_{C(T^\vee)}L^2(\ft)$ maps isometrically into $C(\ft,L^2(\ft))^\Gamma$ via the map
$$\phi\otimes\xi \mapsto \sum_{\alpha\in \Gamma}\phi(x-\alpha)L_\alpha\xi.$$
Moreover this map is surjective.  To see this, note that if $\phi$ is supported inside a single fundamental domain then for $x$ in that fundamental domain we obtain the function $\phi(x)\xi$.  This is extended by equivariance to a function on $\ft$, and using a partition of unity one can approximate an arbitrary element of $C(\ft,L^2(\ft))^\Gamma$ by sums of functions of this form.

We now remark that $C(\ft,L^2(\ft))^\Gamma$ is in fact isomorphic to the Hilbert module $C(T, L^2(\ft))$ via a change of variables. Given $g\in C(\ft,L^2(\ft))^\Gamma$, let $\tilde h(x)=L_{-x}g(x)$. The $\Gamma$-equivariance of $g$ ensures that $g(\gamma+x)=L_\gamma g(x)$ whence
$$\tilde h(\gamma+x)=L_{-x-\gamma}g(\gamma+x)=L_{-x-\gamma}L_\gamma g(x)=L_{-x}g(x)=\tilde h(x).$$
As $\tilde h$ is a $\Gamma$-periodic function from $\ft$ to $L^2(\ft)$ we identify it via the exponential map with the continuous function $h$ from $T$ to $L^2(\ft)$ such that $\tilde h(x)=h(\exp(x))$. Hence $g\mapsto h$ defines the isomorphism $C(\ft,L^2(\ft))^\Gamma\cong C(T,L^2(\ft))$.

We now state the following theorem.

\begin{theorem}\label{Hilbert modules isomorphism}
The Hilbert module $\cE\tensor_{C(T^\vee)}(L^2(\ft)\tensor \cS)$ is isomorphic to $C(T,L^2(\ft)\tensor \cS)$ via the map
$$\phi\otimes (\xi\otimes s) \mapsto \sum_{\alpha\in \Gamma}\phi(x-\alpha)L_{\alpha-x}\xi\otimes s.$$
The representation of $C(T)$ on $L^2(\ft)$ induces a representation $\sigma$ of $C(T)$ on $C(T,L^2(\ft)\tensor \cS)$ defined by
$$[\sigma(f)h](\exp(x),y)=f(\exp(x+y))h(\exp(x),y).$$
Here the notation $h(\exp(x),y)$ denotes the value at the point $y\in\ft$ of $h(\exp(x))\in L^2(\ft)\tensor \cS$.
\end{theorem}

\begin{proof}
We recall that $\cE\tensor_{C(T^\vee)}(L^2(\ft)\tensor \cS)$ is isomorphic to $(\cE\tensor_{C(T^\vee)}L^2(\ft))\tensor \cS$ and we have established that $\cE\tensor_{C(T^\vee)}L^2(\ft)\cong C(T,L^2(\ft))$. This provides the claimed isomorphism.

It remains to identify the representation.  Given $f\in C(T)$ let $\tilde f(x)=f(\exp(x))$ denote the corresponding periodic function on $\ft$. By definition the representation of $C(T)$ on $\cE\tensor_{C(T^\vee)}(L^2(\ft)\tensor \cS)$ takes $\phi\otimes \xi\otimes s$ to $\phi\otimes \tilde f\xi\otimes s$. This is mapped under the isomorphism to the $\Gamma$-periodic function on $\ft$ whose value at $x$ is
$$\sum_{\alpha\in \Gamma}\phi(x-\alpha)L_{\alpha-x}(\tilde f\xi)\otimes s \in L^2(\ft)\tensor \cS.$$
Evaluating this element of $L^2(\ft)\tensor \cS$ at a point $y\in\ft$ we have 
$$\sum_{\alpha\in \Gamma}\phi(x-\alpha)\tilde f(x-\alpha+y)\xi (x-\alpha+y)\otimes s=\tilde f(x+y)\sum_{\alpha\in \Gamma}\phi(x-\alpha)[L_{\alpha-x}\xi](y)\otimes s$$
by $\Gamma$-periodicity of $\tilde f$. Thus $\sigma(f)$ pointwise multiplies the image of $\phi\otimes \xi\otimes s$ in $C(T,L^2(\ft)\tensor \cS)$ by $\tilde f(x+y)=f(\exp(x+y))$ as claimed.
\end{proof}

We now define an operator $\bQ$ on $C(T,L^2(\ft)\tensor \cS)$ by
$$(\bQ h)(\exp(x))=Q_0(h(\exp(x)))$$
for $h\in C(T,L^2(\ft)\tensor \cS)$.

\begin{theorem}
The unbounded operator $\bQ$ is a connection for $Q_0$ in the sense that the bounded operator $\bF=\bQ(1+\bQ^2)^{-1/2}$ is a connection for $F_0=Q_0(1+Q_0^2)^{-1/2}$, after making the identification of Hilbert modules as in Theorem \ref{Hilbert modules isomorphism}.
\end{theorem}

\begin{proof}
Let $Q_x=(L_x\otimes 1_\cS)Q_0(L_{-x}\otimes 1_\cS)$ and correspondingly define
$$F_x=Q_x(1+Q_x^2)^{-1/2}=(L_x\otimes 1_\cS)F_0(L_{-x}\otimes 1_\cS).$$
The commutators $[L_x\otimes 1_\cS,Q_0]$ are bounded (the argument is exactly as for $[L_\gamma\otimes 1_\cS,Q_0]$ in Section \ref{K-homology element}). It follows (in the spirit of Baaj-Julg, \cite{BJ}) that the commutators $[L_x\otimes 1_\cS,F_0]$ are compact. Thus $F_x-F_0$ is a compact operator for all $x\in \ft$.

To show that $\bF$ is a connection for $F_0$ we must show that for $\phi\in \cE$, the diagram
$$\begin{CD}
L^2(\ft)\tensor \cS @>F_0>> L^2(\ft)\tensor \cS\\
@V\phi\otimes VV @V\phi\otimes VV\\
\cE\otimes L^2(\ft)\tensor \cS @. \cE\otimes L^2(\ft)\tensor \cS\\
@V\cong VV @V\cong VV\\
C(T,L^2(\ft)\tensor \cS) @>>\bF> C(T,L^2(\ft)\tensor \cS)
\end{CD}$$
commutes modulo compact operators.

Following the diagram around the right-hand side we have 
$$\xi\otimes s \mapsto \sum_{\alpha\in\Gamma}\phi(x-\alpha)(L_{\alpha-x}\otimes 1_\cS)F_0(\xi\otimes s)$$
while following the left-hand side we have
$$\bF\Big[\sum_{\alpha\in\Gamma}\phi(x-\alpha)(L_{\alpha-x}\otimes 1_\cS)(\xi\otimes s)\Big]=\sum_{\alpha\in\Gamma}\phi(x-\alpha)F_0(L_{\alpha-x}\otimes 1_\cS)(\xi\otimes s).$$
As $[F_0,L_{\alpha-x}\otimes 1_\cS]$ is a compact operator for each $x$ and the sum is finite for each $x$, the difference between the two paths around the diagram is a function from $T$ to compact operators on $L^2(\ft)\tensor\cS$. It is thus a compact operator from the Hilbert space $L^2(\ft)\tensor\cS$ to the Hilbert module $C(T,L^2(\ft)\tensor \cS)$ as required.
\end{proof}

\begin{theorem}\label{first Kasparov product}
The Kasparov product $\cP\otimes_{C(T^\vee)}\cQ$ is the identity $1_{C(T)}$ in $KK_W(C(T),C(T))$.
\end{theorem}

\begin{proof}
We define a homotopy of representations of $C(T)$ on $C(T,L^2(\ft)\tensor \cS)$ by
$$[\sigma_\lambda(f)h](\exp(x),y)=f(\exp(x+\lambda y))h(\exp(x),y)$$
and note that $\sigma_1=\sigma$ while $\sigma_0$ is simply the representation of $C(T)$ on $C(T,L^2(\ft)\tensor \cS)$ by pointwise multiplication of functions on $T$. It is easy to see that these representations are $W$-equivariant.

Let $f$ be a smooth function on $T$ and let $h\in C(T,L^2(\ft)\tensor \cS)$. Let $\tilde{f}(x)=f(\exp(x))$ and let $\tilde{h}(x,y)=h(\exp(x),y)$.  Then
\begin{align*}
([\bQ,\sigma_\lambda(f)]h)&(\exp(x),y)\\
=\,\,&\big[\frac{\partial}{\partial y^j}(\eps^j \tilde f(x+\lambda y)\tilde h(x,y)) - 2\pi i y^j\e_j\tilde f(x+\lambda y)\tilde h(x,y)\big] \\
&-\big[\tilde f(x+\lambda y)\frac{\partial}{\partial y^j}(\eps^j \tilde h(x,y)) - \tilde f(x+\lambda y)2\pi i y^j\e_j\tilde h(x,y)\big]\\
=\,\,&\frac{\partial}{\partial y^j}(\tilde f(x+\lambda y))(\eps^j \tilde h(x,y)).
\end{align*}
For each $\lambda$ the operator $\bQ$ thus commutes with the representation $\sigma_\lambda$ modulo bounded operators on a dense subalgebra of $C(T)$. Hence for each $\lambda$ $(C(T,L^2(\ft)\tensor \cS),\sigma_\lambda,\bQ)$ defines an unbounded Kasparov triple.

This is true in particular for $\lambda=1$ and thus $(C(T,L^2(\ft)\tensor \cS),\sigma,\bQ)$ is a Kasparov triple so as the operator in the triple $\cP$ is zero while $\bQ$ is a connection for $Q_0$ it follows that $\cP\otimes_{C(T^\vee)} \cQ=(C(T,L^2(\ft)\tensor \cS),\sigma,\bQ)$ in $KK_W(C(T),C(T))$.

Now applying the homotopy we have $\cP\otimes_{C(T^\vee)} \cQ=(C(T,L^2(\ft)\tensor \cS),\sigma_0,\bQ)$. Since $\sigma_0$ commutes exactly with the operator $\bQ$ the representation$\sigma_0$ respects the direct sum decomposition of $C(T,L^2(\ft)\tensor \cS)$ as $C(T,\ker(Q_0))\oplus C(T,\ker(Q_0)^\perp)$. The operator $\bQ$ is invertible on the second summand (and commutes with the representation) and hence the corresponding Kasparov triple $(C(T,\ker(Q_0)^\perp),\sigma_0|_{C(T,\ker(Q_0)^\perp)},\bQ|_{C(T,\ker(Q_0)^\perp)})$ is zero in $KK$-theory.

We thus conclude that $\cP\otimes_{C(T^\vee)} \cQ=(C(T,\ker(Q_0)),\sigma_0|_{C(T,\ker(Q_0))},0)$. Since $\ker Q_0$ is 1-dimensional (Theorem \ref{ladder}) the module $C(T,\ker(Q_0))$ is isomorphic to $C(T)$ and the restriction of $\sigma_0$ to this is the identity representation of $C(T)$ on itself.  Thus $\cP\otimes_{C(T^\vee)} \cQ=(C(T),1,0)=1_{C(T)}$.
\end{proof}

\subsection{The Kasparov product $\cP\otimes_{C(T)}\cQ$}

We begin by considering the dual picture, which exchanges the roles of $T$ and $T^\vee$. There exist elements $\cQ^\vee \in KK_W(C(T)\tensor C(T^\vee),\C)$ and $\cP^\vee\in KK_W(\C,C(T^\vee)\tensor C(T))$ for which the result of the previous section implies $\cP^\vee\otimes_{C(T)}\cQ^\vee=1_{C(T^\vee)}$ in $KK_W(C(T^\vee),C(T^\vee))$.

We will show that there is an isomorphism
$$\theta:C(T^\vee)\tensor C(T)\to C(T)\tensor C(T^\vee)$$
such that $\cQ=\theta^*\cQ^\vee$ and $\cP=\theta^{-1}_*\cP^\vee$. This will imply that $\cP\otimes_{C(T)}\cQ=\cP^\vee\otimes_{C(T)}\cQ^\vee=1_{C(T^\vee)}$ in $KK_W(C(T^\vee),C(T^\vee))$ and hence will complete the proof of the Poincar\'e duality between $C(T)$ and $C(T^\vee)$.

We recall that $\cQ$ is represented by the (unbounded) Kasparov triple $(L^2(\ft)\tensor \cS,\rho,Q_0)$ where $\cS=\cl(\ft\times\ft^*)P$, for $P$ the projection $P=\prod_{j}\frac 12(1-i\e_j\eps^j)$ and
$$Q_0=\frac{\partial}{\partial y^j}\otimes \eps^j-2\pi i y^j\otimes \e_j.$$
For $\gamma\in \Gamma$, $\chi\in\Gamma^\vee$ and correspondingly $e^{2\pi i\la \eta,\gamma\ra}$ in $C(T^\vee)$, $e^{2\pi i\la \chi,x\ra }$ in $C(T)$, the representation $\rho$ of $C(T^\vee)\tensor C(T)$ is defined by
$$\rho(e^{2\pi i \la \eta,\gamma\ra})(\xi\otimes s)=L_\gamma\xi\otimes s, \text { and }\rho(e^{2\pi i\la \chi,x\ra})(\xi\otimes s)=e^{2\pi i\la \chi,x\ra}\xi \otimes s.$$

By definition $\cQ^\vee$ is represented by the triple $(L^2(\ft^*)\tensor \cS^\vee,\rho^\vee,Q_0^\vee)$ where $\cS^\vee=\cl(\ft^*\times\ft)P^\vee$, for $P^\vee$ the projection $P^\vee=\prod_{j}\frac 12(1-i\eps^j\e_j)$ and
$$Q_0^\vee=\frac{\partial}{\partial \eta_j}\otimes \e_j-2\pi i \eta_j\otimes \eps^j.$$
For $\gamma\in \Gamma$, $\chi\in\Gamma^\vee$ and correspondingly $e^{2\pi i\la \eta,\gamma\ra}$ in $C(T^\vee)$, $e^{2\pi i\la \chi,x\ra }$ in $C(T)$, the representation $\rho^\vee$ of $C(T)\tensor C(T^\vee)$ is now defined by
$$\rho^\vee(e^{2\pi i\la \chi,x\ra})(\xi^\vee\otimes s^\vee)=L^\vee_\chi\xi^\vee \otimes s^\vee, \text { and }\rho^\vee(e^{2\pi i \la \eta,\gamma\ra})(\xi^\vee\otimes s^\vee)=e^{2\pi i \la \eta,\gamma\ra}\xi^\vee\otimes s^\vee.$$
Here $L^\vee_\chi$ denotes the translation action of $\chi\in\Gamma^\vee$ on $L^2(\ft^*)$.

In our notation, $\eps^j$ is again an orthonormal basis for $\ft^*$ and $\e_j$ is an orthonormal basis for $\ft$. We can canonically identify $\cl(\ft\times\ft^*)$ with $\cl(\ft^*\times\ft)$, and hence think of both $\cS$ and $\cS^\vee$ as subspaces of this algebra.

\bigskip

We can identify $L^2(\ft)$ with $L^2(\ft^*)$ via the Fourier transform: let $\cF:L^2(\ft)\to L^2(\ft^*)$ denote the Fourier transform isomorphism
$$[\cF\xi](\eta)=\int_\ft \xi(y)e^{2\pi i\la \eta,y\ra}\,dy.$$
It is easy to see that this is $W$-equivariant.

To identify $\cS$ with $\cS^\vee$, let $u\in \cl(\ft\times\ft^*)$ be defined by $u=\eps^1\eps^2\dots\eps^n$ when $n=\dim(\ft)$ is even and $u=e_1e_2\dots e_n$ when $n$ is odd.

\begin{lemma}
Conjugation by $u$ defines a $W$-equivariant unitary isomorphism $\cU:\cS\to\cS^\vee$. For $a\in\cl(\ft\times\ft^*)$ (viewed as an operator on $\cS$ by Clifford multiplication) $\cU a\cU^*$ is Clifford multiplication by $uau^*$ on $\cS^\vee$ and in particular $\cU \e_j\cU^*=\e_j$ while $\cU \eps^j\cU^*=-\eps^j$.
\end{lemma}

\begin{proof}
We first note that $u$ respectively commutes and anticommutes with $\e_j$, $\eps^j$ (there being respectively an even or odd number of terms in $u$ which anticommute with $\e_j$, $\eps^j$). It follows that $uPu^*=P^\vee$, hence conjugation by $u$ maps $\cS$ to $\cS^\vee$.

Denoting by $\pi:\C P\to \C$ the identification of $\C P$ with $\C$, the inner product on $\cS$ is given by $\la s_1,s_2\ra= \pi(s_1^*s_2)$ while the inner product on $\cS^\vee$ is given by $\la s_1^\vee,s_2^\vee\ra= \pi(u^*(s_1^\vee)^*s_2^\vee u)$. Thus
$$\la usu^*,s^\vee\ra=\pi(u^*(usu^*)^*s^\vee u)=\pi(s^*u^*s^\vee u)=\la s,u^*s^\vee u\ra$$so $\cU^*$ is conjugation by $u^*$ which inverts $\cU$ establishing that $\cU$ is unitary.

We now check that $\cU$ is $W$-equivariant. In the case that $\ft$ is even-dimensional, we note that identifying $\cl(\ft^*)$ with the exterior algebra of $\ft^*$ (as a $W$-vector space), $u$ corresponds to the volume form on $\ft^*$ so $w\cdot u=\det (w)u$. Similarly in the odd dimensional case $u$ corresponds to the volume form on $\ft$ and again the action of $w$ on $u$ is multiplication by the determinant. Thus
$$w\cdot\cU(s)=w\cdot (usu^*)=(w\cdot u)(w\cdot s)(w\cdot u^*)=\det(w)^2\,u(w\cdot s)u^*=\cU(w\cdot s)$$
since $\det(w)=\pm1$.

Finally for $s^\vee\in\cS^\vee$ and $a\in\cl(\ft\times\ft^*)$ we have
$$\cU a\cU^*s^\vee=\cU(au^*s^\vee u)=uau^*s^\vee$$
and hence $\cU \e_j\cU^*=u\e_ju^*=\e_j$, $\cU \eps^j\cU^*=u\eps^ju^*=-\eps^j$.
\end{proof}

Since $\cF\otimes\cU$ is a $W$-equivariant unitary isomorphism from $L^2(\ft)\tensor \cS$ to $L^2(\ft^*)\tensor \cS^\vee$, the triple $(L^2(\ft)\tensor \cS,\rho,Q_0)$ representing $\cQ$ is isomorphic to the Kasparov triple
$$(L^2(\ft^*)\tensor \cS^\vee,(\cF\otimes \cU)\rho(\cF^*\otimes \cU^*),(\cF\otimes u)Q_0(\cF^*\otimes \cU^*)).$$

\begin{theorem}\label{theta^* fa^vee}
Let $\theta:C(T^\vee)\tensor C(T)\to C(T)\tensor C(T^\vee)$ be defined by
$$\theta(g\otimes f)=f\otimes (g\circ \epsilon).$$
where $\epsilon$ is the involution on $T^\vee$ defined by $\epsilon(\exp(\eta))=\exp(-\eta)$. Then $\cQ=\theta^*\cQ^\vee$ in $KK_W(C(T^\vee)\tensor C(T),\C)$.
\end{theorem}

\begin{proof}
We will show that $\rho^\vee\circ\theta=(\cF\otimes \cU)\rho(\cF^*\otimes \cU^*)$ and $(\cF\otimes u)Q_0(\cF^*\otimes \cU^*)=Q_0^\vee$. We begin with the operator.

The operator $Q_0$ is given by
$$\frac{\partial}{\partial y^j}\otimes \eps^j-2\pi i y^j\otimes \e_j.$$
Conjugating the operator $\frac{\partial}{\partial y^j}$ by the Fourier transform we obtain the multiplication by $2\pi i\eta_j$, while conjugating $-2\pi i y^j$ by the Fourier transform we obtain the multiplication by $-2\pi i(\frac i{2\pi}\frac{\partial}{\partial \eta_j})=\frac{\partial}{\partial \eta_j}$. Conjugation by $\cU$ negates $\eps^j$ and preserves $\e_j$ hence
$$(\cF\otimes u)Q_0(\cF^*\otimes \cU^*)=2\pi i \eta_j\otimes (-\eps^j)+\frac{\partial}{\partial \eta_j}\otimes \e_j=Q_0^\vee.$$
For the representation, $\rho(e^{2\pi i\la \chi,x\ra})$ is multiplication by $e^{2\pi i\la \chi,x\ra}$ on $L^2(\ft)$ (with the identity on $\cS$) and conjugating by the Fourier transform we get the translation $L^\vee_\chi$, hence $(\cF\otimes \cU)\rho(e^{2\pi i\la \chi,x\ra})(\cF^*\otimes \cU^*)=\rho^\vee(e^{2\pi i\la \chi,x\ra})$.  On the other hand $\rho(e^{2\pi i \la \eta,\gamma\ra})$ is the translation $L_\gamma$ and Fourier transforming we get the multiplication by $e^{-2\pi i \la \eta,\gamma\ra}$. Thus  $(\cF\otimes \cU)\rho(e^{2\pi i \la \eta,\gamma\ra})(\cF^*\otimes \cU^*)=\rho^\vee(e^{2\pi i \la -\eta,\gamma\ra})$.

We conclude that $(\cF\otimes \cU)\rho(\cF^*\otimes \cU^*)=\rho^\vee\circ\,\theta$ as required.
\end{proof}

\begin{theorem}\label{second Kasparov product}
The Kasparov product $\cP\otimes_{C(T)}\cQ$ is $1_{C(T^\vee)}$ in the Kasparov group $KK_W(C(T^\vee),C(T^\vee))$.
\end{theorem}

\begin{proof}
We have $\cP\otimes_{C(T^\vee)}\cQ=1_{C(T)}$ in $KK_W(C(T),C(T))$ by Theorem \ref{first Kasparov product} while  $\cP^\vee\otimes_{C(T)}\cQ^\vee=1_{C(T^\vee)}$ in $KK_W(C(T^\vee),C(T^\vee))$ by  Theorem \ref{first Kasparov product} for the dual group.

By Theorem \ref{theta^* fa^vee} we have $\cQ^\vee=(\theta^{-1})^*\cQ$ whence
$$1_{C(T^\vee)}=\cP^\vee\otimes_{C(T)}\cQ^\vee=(\theta^{-1})_*\cP^\vee\otimes_{C(T)} \cQ.$$
Let $\cP'=(\theta^{-1})_*\cP^\vee$ in $KK_W(\C,C(T)\tensor C(T^\vee))$. Then
$$\cP=\cP\otimes_{C(T^\vee)}1_{C(T^\vee)}=\cP\otimes_{C(T^\vee)}(\cP'\otimes_{C(T)} \cQ).$$
By definition $\cP'\otimes_{C(T)} \cQ=(\cP'\otimes 1_{C(T^\vee)})\otimes_{_{\scriptstyle C(T)\otimes C(T^\vee)}} \cQ$ \,and hence 
$$\cP=(\cP\otimes\cP')\otimes_{_{\scriptstyle C(T^\vee)\tensor C(T)}} \cQ$$
by associativity of the Kasparov product. Here $\cP\otimes\cP'$ is the `external' product and lives in $KK_W(\C,C(T)\tensor C(T)\tensor C(T^\vee)\tensor C(T^\vee))$, with $\cP$ appearing in the first and last factors, and $\cP'$ in the second and third. The product with $\cQ$ is over the second and last factors. Similarly
$$\cP'=\cP'\otimes_{C(T)}(\cP\otimes_{C(T^\vee)} \cQ)=(\cP'\otimes\cP)\otimes_{_{\scriptstyle C(T)\tensor C(T^\vee)}} \cQ$$
where $\cP'$ now appears as the first and last factors and the product with $\cQ$ is over the first and third factors. Up to reordering terms of the tensor product $(\cP\otimes\cP')\otimes_{_{\scriptstyle C(T^\vee)\tensor C(T)}} \cQ=(\cP'\otimes\cP)\otimes_{_{\scriptstyle C(T)\tensor C(T^\vee)}} \cQ$.

Thus (by commutativity of the external product) $\cP=\cP'=(\theta^{-1})_*\cP^\vee$ and hence $\cP\otimes_{C(T)}\cQ=1_{C(T^\vee)}$. This completes the proof.
\end{proof}

\begin{corollary}
The elements $\cQ\in KK_W(C(T^\vee)\tensor C(T),\C)$ and $\cP\in KK_W(\C,C(T)\tensor C(T^\vee))$ exhibit a $W$-equivariant Poincar\'e duality between the algebras $C(T)$ and $C(T^\vee)$.
\end{corollary}

\section{Poincar\'e duality between $C_0(\ft)\rtimes (\Gamma\rtimes W)$ and $C_0(\ft^*)\rtimes (\Gamma^\vee\rtimes W)$}

\subsection{Descent of Poincar\'e duality}\label{descent}
For $W$ a group, a Poincar\'e duality between two $W$-$C^*$-algebras $A,B$ induces a natural family of isomorphisms
\[
KK_W(A\tensor D_1,D_2)\cong KK_W(D_1,B\tensor D_2)
\]
for $W$-$C^*$-algebras $D_1,D_2$.  In other words the functor $A\tensor$ is left-adjoint to $B\tensor$ on the $KK_W$ category when there is a Poincar\'e duality from $A$ to $B$. (The symmetry of Poincar\'e dualities means that $B\tensor$ is also left-adjoint to $A\tensor$ ).  The element in $KK_W(\C,A\tensor B)$ defining the Poincar\'e duality is precisely the unit of the adjunction, while the counit is given by the element in $KK_W(B\tensor A,\C)$.  This categorical view of Poincar\'e duality appears in \cite{EEK, E, EM}.

Now let $D_1,D_2$ be $C^*$-algebras (without $W$-action).  Let $\tau$ denote the trivial-action functor from $KK$ to $KK_W$, i.e. $\tau D_1$, $\tau D_2$ are $W$-$C^*$=algebras with trivial action of $W$. In the case that $W$ is a \emph{finite} group a Poincar\'e duality yields isomorphisms
\begin{align*}
KK(A\rtime W\tensor D_1,D_2)&\cong KK_W(A\tensor \tau D_1,\tau D_2)\\&\cong KK_W(\tau D_1,B\tensor \tau D_2)\\&\cong KK(D_1,B\rtime W\tensor D_2).
\end{align*}
The first and last isomorphisms are the dual Green-Julg and Green-Julg isomorphisms respectively, and in categorical terms these amount to the fact that the $\tau$ functor is (right and left) adjoint to the descent functor $\rtime W$, see \cite{M}. We denote the unit and counit by $\alpha$ and $\beta$, for the left-adjunction from $\tau$ to $\rtime W$, and by $\widehat\alpha$ and $\widehat\beta$ for the left adjunction from $\rtime W$ to $\tau$.

Since this is natural $A\rtime W\tensor$ is left-adjoint to $B\rtime W\tensor$, hence there must exist a unit and a counit providing this descended Poincar\'e duality. We will identify these elements explicitly.

\begin{theorem}\label{descent of PD}
Let $\fa\in KK_W(B\tensor A,\C)$ and $\fb\in KK_W(\C,A\tensor B)$ define a $W$-equivariant Poincar\'e duality between $W$-$C^*$-algebras $A,B$, with $W$ finite. Then  $\widetilde\fa=\fTr(\fa\rtime W)\beta_\C\in KK(B\rtime W\,\tensor\, A\rtime W,\C)$ and $\widetilde\fb=\alpha_\C (\fb\rtimes W)\Delta\in KK(\C,A\rtime W\,\tensor\, B\rtime W)$ define a Poincar\'e duality such that the following diagram commutes
\[
\begin{CD}
KK^*_W(A,\C)@>>{\fb\tensor\!\!_A\,\text{---}}>KK^*_W(\C,B)\\
@VV\text{dual Green-Julg $\cong$}V@VV\text{Green-Julg $\cong$}V\\
KK^*(A\rtime W,\C)@>>{\widetilde\fb\tensor\!\!_{A\rtime W}\,\text{---}}>KK^*(\C,B\rtime W).
\end{CD}
\]
\end{theorem}

Here $\Delta\in KK((A\tensor B)\rtime W, A\rtime W\,\tensor\, B\rtime W)$ is given by the diagonal inclusion of $W$ into $W\times W$ and $\fTr\in  KK(B\rtime W\,\tensor\, A\rtime W,(B\tensor A)\rtime W)$ is dual to this: We define a positive linear map $\Tr:(B\tensor A)\rtimes(W\times W)\to (B\tensor A)\rtimes W$ by
\[
\Tr: (a\otimes b)[w_1,w_2]\mapsto \begin{cases}(a\otimes b)[w_1] &\text{ if }w_1=w_2\\ 0& \text{ otherwise.}\end{cases}
\]
This is a $(B\tensor A)\rtimes W$-module map and we equip the algebra $(B\tensor A)\rtimes(W\times W)$ with inner product in $(B\tensor A)\rtimes W$ defined by
\[
\left\la (b\otimes a)[w_1,w_2],(b'\otimes a')[w_1',w_2']\right\ra_{\Tr}=\Tr\left([w_1^{-1},w_2^{-1}](b^*\otimes a^*)(b'\otimes a')[w_1',w_2']\right).
\]
The completion of this as a Hilbert module, equipped with the left multiplication representation of $(B\tensor A)\rtimes(W\times W)$ provides the required element $\fTr\in KK((B\tensor A)\rtimes(W\times W),(B\tensor A)\rtimes W)$.

\bigskip

To identify the unit and counit $\widetilde\fb$ and $\widetilde\fa$ one proceeds as follows. The unit $\widetilde\fb$ is the image of the identity $1\!_{A\rtime W}$ under the isomorphism $KK(A\rtime W,A\rtime W)\cong KK(\C,B\rtime W\,\tensor\, A\rtime W)$. This is the composition of the dual Green-Julg, equivariant Poincar\'e duality, and Green-Julg maps.  The first two yield the Poincar\'e dual of the unit $\widehat{\alpha}_A$. One must then descend this and pair with the unit $\alpha_\C$. Hence $\widetilde\fb=\alpha_\C(b(\widehat{\alpha}_A\otimes 1_B))\rtime W=\alpha_\C(b\rtime W)((\widehat{\alpha}_A\otimes 1_B)\rtime W)$ by naturality of descent.  It is not hard to identify $(\widehat{\alpha}_A\otimes 1_B)\rtime W$ as the element $\Delta$.

Similarly the counit $\widetilde\fa$ is the image of the identity $1\!_{B\rtime W}$ under the isomorphism $KK(B\rtime W,B\rtime W)\cong KK(A\rtime W\,\tensor\, B\rtime W,\C)$. This is the composition of the Green-Julg, equivariant Poincar\'e duality, and dual Green-Julg maps, hence $\widetilde\fa$ is obtained by taking the Poincar\'e dual of the counit ${\beta}_B$, descending, and applying the counit  $\widehat{\beta}_\C$. We have $\widetilde\fa=((\widehat{\beta}_B\otimes 1_A)a)\rtime W)\widehat{\beta}_\C=((\widehat{\beta}_B\otimes 1_A)\rtime W)(a\rtime W)\widehat{\beta}_\C$. A change of variables identifies $(\widehat{\beta}_B\otimes 1_A)\rtime W$ with $\fTr$.

\begin{remark}\label{btilde remark}
Given a Kasparov triple $(\cE,1,D)$ representing $\fb$ we can describe explicitly a triple $(\widetilde\cE,\widetilde\alpha_\C,D\otimes 1)$ for $\widetilde{\fb}$.

The module $\widetilde{\cE}$ is given by descending $\cE$ and inflating the action of $W$ to $W\times W$. Explicitly $\widetilde{\cE}$  is the completion of $\cE\tensor \C[W\times W]$ with respect to the inner product
$$\la \xi\otimes [w_1,w_2],\xi'\otimes [w_1',w_2']\ra=(w_1^{-1},w_2^{-1})\cdot\la \xi,\xi'\ra[w_1^{-1}w_1',w_2^{-1}w_2'].$$
The operator is simply $D\otimes 1$ on $\widetilde\cE$.

The representation $\widetilde\alpha_\C$ of $\C$ on $\widetilde\cE$ takes $1$ to the projection corresponding to the trivial representation of $W$, where $W$ acts diagonally on $\widetilde\cE$ -- the unit $\alpha_\C\in KK(\C,(\tau\C)\rtime W)$ is given by inclusion of $\C$ as the trivial representation in $\C[W]=(\tau \C)\rtime W$.
\end{remark}

\subsection{Proof of Theorem \ref{PD}}\label{Proof of TheoremPD}

Theorem \ref{PD} follows from Theorem \ref{descent of PD} by consideration of the following diagram.

\[
\begin{CD}
KK^*_W(C(T),\C) @>{\cP\tensor\!\!_{C(T)}\,\text{---}}>> KK^*_W(\C,C(T^\vee)).\\
@V\cong V\text{Morita}V @V\cong V\text{Morita} V \\
KK^*_W(C_0(\ft)\rtimes \Gamma,\C) @>{\fb\tensor\!\!_{C_0(\ft)\rtimes \Gamma}\,\text{---}}>> KK^*_W(\C,C_0(\ft^*)\rtimes\Gamma^\vee)\\
@V\cong V\text{dual Green-Julg}V @V\cong V\text{Green-Julg}V \\
KK^*(C_0(\ft)\rtimes (\Gamma\rtimes W),\C) @>{\widetilde\fb\tensor\!\!_{C_0(\ft)\rtimes (\Gamma\rtimes W)}\,\text{---}}>> KK^*(\C,C_0(\ft^*)\rtimes(\Gamma^\vee\rtimes W))\\
\end{CD}
\]

Composition of $\cP$ with the Morita equivalences and of $\cQ$ with the inverse Morita equivalences, yields a $W$-equivariant Poincar\'e duality between $C_0(\ft)\rtimes \Gamma$ and $C_0(\ft^*)\rtimes\Gamma^\vee$ inducing the middle arrow. 

To determine the element $\fb$ explicitly, recall that  $\cP$ is given by the Hilbert module of functions $\sigma: \ft\times \ft^*\rightarrow \C$ which are $\Gamma^\vee$ periodic in the second variable and satisfying 
\[
\sigma(\gamma+x, \eta)=e^{2\pi i\la \eta,\gamma\ra} \sigma(x, \eta).
\]
This module is equipped with the inner product
$$\la \sigma_1, \sigma_2\ra(x,\eta)= \ol{\sigma_1(x,\eta)}\sigma_2(x,\eta).$$

The Morita equivalence from $C(T)$ to $C_0(\ft)\rtimes\Gamma$ is given by the completion of $C_c(\ft)$ with respect to the inner product 
\[
\la \phi_1,\phi_2\ra=\sum\limits_{\gamma\in \Gamma}\ol{\phi_1}\,(\gamma\cdot \phi_2)[\gamma],
\]
and similarly for $C(T^\vee)$.

It follows that $\fb$ is given by the Hilbert module completion of $C_c(\ft\times \ft^*)$ with respect to the inner product

\[
\la \theta_1,\theta_2\ra=\sum\limits_{(\gamma,\chi) \in \Gamma\times \Gamma^\vee}\ol{\theta_1}\,\left((\gamma,\chi)\cdot\theta_2\right)
e^{2\pi i\la \eta,\gamma\ra} [(\gamma,\chi)].
\]

Applying Theorem \ref{descent of PD} yields the bottom arrow. Here we identify $(C_0(\ft)\rtimes \Gamma)\rtimes W$ with $C_0(\ft)\rtimes (\Gamma\rtimes W)$ and $(C_0(\ft^*)\rtimes\Gamma^\vee)\rtimes W)$ with $C_0(\ft^*)\rtimes(\Gamma^\vee\rtimes W)$. As noted in Remark \ref{btilde remark} the element $\widetilde \fb$ has Hilbert module obtained by descending the module and inflating the $W$ action to $W\times W$. 

In conclusion we obtain the module by completing $C_c(\ft\times \ft^*)\rtimes (W\times W)$ with respect to the inner product 
\[
\la \theta[w_1, w_2], \theta'[w_1',w_2']\ra= (w_1,w_2)^{-1}\cdot \la\theta,\theta'\ra[w_1^{-1}w_1',w_2^{-1}w_2'],
\]
where $\la\theta,\theta'\ra$ is the inner product on $C_c(\ft\times\ft^*)$ defined above which is equipped with the representation of $\C$ given by the trivial projection in $\C[W]$, where $W$ acts diagonally on all factors.

\subsection{Proof of Theorem \ref{PD2}}

Theorem \ref{PD2} follows from Theorem \ref{descent of PD} by the consideration of the following diagram.

\[
\begin{CD}
KK^*_W(C(T),\C) @>{\cP\tensor\!\!_{C(T)}\,\text{---}}>> KK^*_W(\C,C(T^\vee)).\\
@V\cong V\text{Fourier-Pontryagin}V @V\cong V\text{Fourier-Pontryagin} V \\
KK^*_W(C^*(\Gamma^\vee),\C) @>{\fb\tensor\!\!_{C^*(\Gamma^\vee)}\,\text{---}}>> KK^*_W(\C,C^*(\Gamma))\\
@V\cong V\text{dual Green-Julg}V @V\cong V\text{Green-Julg}V \\
KK^*(C^*(\Gamma^\vee\rtimes W),\C) @>{\widetilde\fb\tensor\!\!_{C^*(\Gamma^\vee\rtimes W)}\,\text{---}}>> KK^*(\C,C^*(\Gamma\rtimes W))\\
\end{CD}
\]

\medskip
\noindent Composition of $\cP$ and of $\cQ$ with the Fourier-Pontryagin isomorphisms yields a $W$-equivariant Poincar\'e duality between $C^*(\Gamma^\vee)$ and $C^*(\Gamma)$ inducing the middle arrow. Applying Theorem \ref{descent of PD} yields the bottom arrow. Here we identify $C^*(\Gamma^\vee)\rtimes W$ with $C^*(\Gamma^\vee\rtimes W)$ and $C^*(\Gamma) \rtimes W$ with $C^*(\Gamma \rtimes W)$.

\subsection{The connection with the Baum Connes assembly map}\label{baumconnesconnection}\ 

The Poincar\'e duality isomorphism  appearing in Theorem \ref{PD}
\[
KK^*(C_0(\ft)\rtimes (\Gamma\rtimes W),\C)\longrightarrow KK^*(\C,C_0(\ft^*)\rtimes(\Gamma^\vee\rtimes W))
\]
can be identified with the Baum Connes assembly map for the group  $\Gamma\rtimes W$ in a sense made explicit by the following diagram. Note that while we have suppressed the indices, these are degree $0$ maps of $\Z_2$-graded groups. 
\bigskip

\noindent\begin{small}
\begin{tikzcd}
KK_{\Gamma\rtime W}(C_0(\ft), \C)
\arrow{r}{\textrm{Baum-Connes}}
\arrow[bend right=75]{dddd}{}
\arrow{dd}{\textrm{dual Green-Julg}}
&
KK(\C, C^*(\Gamma\rtime W))
\arrow{dd}{\textrm{Morita equivalence}}
\\
\\
KK(C_0(\ft)\rtime(\Gamma\rtime W),\C)
\arrow{r}{\textrm{Poincar\'e duality}}
&
KK(\C,C_0(\ft^*)\rtime(\Gamma^\vee\rtime W))
\\
\\
KK(C_0(\ft)\rtime(\Gamma\rtime W), C^*(\Gamma\rtime W))
\arrow{uu}{\times\widehat\beta_\C}
\arrow{r}{\textrm{P. d.}}
&
KK(\C,C_0(\ft^*)\rtime(\Gamma^\vee\rtime W) \otimes C^*(\Gamma\rtime W)) 
\arrow{uu}{\times\widehat\beta_\C}
\\
\end{tikzcd}
\end{small}

The curved arrow is the descent map.
Note that since $\Gamma\rtimes W$ is amenable the full and reduced $C^*$-algebras agree. The counit $\widehat\beta_\C\in KK((\tau\C)\rtime \Gamma \rtime W),\C)=KK(C^*(\Gamma\rtime W),\C)$ is given by the trivial representation of the group $\Gamma\rtimes W$. This element  has the effect of collapsing the coefficients $C^*(\Gamma\rtimes W)$. 

The upper and lower Poinca\'re dualities in the diagram are both provided by Theorem \ref{PD}, in the lower case with the coefficients $C^*(\Gamma\rtimes W)$, and the element inducing the map from $K$-homology to $K$-theory is described in detail in Section \ref{Proof of TheoremPD}.

Clearly the lower square commutes by associativity of the Kasparov product, while the left hand triangle commutes by definition. Therefore, to show that the Baum Connes assembly map corresponds to the upper Poincar\'e duality it suffices to show that the outer pentagon is commutative.

By definition the assembly map is the composition of descent with a Kasparov product. We will use the notation $\alphaCt$ to denote the relevant element of $KK(\C, C_0(\ft)\rtimes\Gamma\rtimes W)$, which is given by the Hilbert module obtained by completing $C_c(\ft)$ with respect to the inner product

\[
\langle f, f'\rangle = \sum\limits_{\gamma, w}\overline{f}((\gamma w)\cdot f)[\gamma w].
\]

We thus have the following diagram, where the bottom arrow is our Poincar\'e duality.
\bigskip
 
\noindent\hskip-.7em\begin{small}
\begin{tikzcd}%[column sep=large]
KK_{\Gamma\rtimes W}(C_0(\ft), \C)
\arrow{r}{\textrm{Baum-Connes}}
\arrow[swap]{dddd}{\textrm{descent}}
&
KK(\C, C^* (\Gamma\rtimes W))
\arrow{dd}{\textrm{Morita equivalence}}
\\
&
\\
&
KK(\C,C_0(\ft^*)\rtimes(\Gamma^\vee\rtimes W))
\\
&
\\
KK(C_0(\ft)\rtimes(\Gamma\rtimes W), C^*(\Gamma\rtimes W))
\arrow{ruuuu}{\mbox{\raisebox{3ex}{\LARGE$\circlearrowleft$}}\quad
\alphaCt\times}[swap]{\mbox{\quad\strut\raisebox{-3ex}{\LARGE$\cancel{\circlearrowleft}$}}}
\arrow{r}{}
&
KK(\C,C_0(\ft^*)\rtimes(\Gamma^\vee\rtimes W) \otimes C^*(\Gamma\rtimes W))
\arrow{uu}{\times\widehat\beta_\C}
\\
\end{tikzcd}
\end{small}

\noindent The upper triangle commutes by definition of the assembly map, however, it should be noted that the lower quadrilateral does not commute: the two directions around the quadrilateral collapse different algebras.  It is thus not entirely obvious that the outer pentagon itself commutes. However we will show that the quadrilateral does commute on the image of the descent map so that the outer pentagon commutes as required.

We start with a Kasparov cycle $(H,\rho, T)\in KK_{\Gamma\rtimes W}(C_0(\ft), \C)$. Note that since the action of $\Gamma\rtimes W$ on $\ft$ is proper we may, without loss of generality, take $T$ to be exactly invariant and of finite propagation. Now we descend to get $(\cE, \hat\rho, T\otimes 1)$, where $\cE = H\tensor C^*\Gamma\rtimes W$ and $\hat\rho$ is a representation defined by $\hat\rho(f[g])=\rho(f)\pi(g)\otimes [g]$, ($\pi$ denotes the representation of $\Gamma\rtimes W$ on $H$).

Applying our Poincar\'e duality, given by the completion of $C_c(\ft\times \ft^*) \rtimes (\Gamma\times\Gamma^\vee)$ described in Section \ref{Proof of TheoremPD}, along with with representation of $\C$ given by the trivial representation of $W$,  we obtain a Kasparov triple as follows:

Let $H_c=\rho(C_c(\ft))H$. The module in our triple is the completion of $H_c\tensor C_c(\ft^*)\tensor\C[(\Gamma\rtimes W)\times W]$ with respect to the inner product 
\[
\la\xi\otimes f[(g,w)],\xi'\otimes f'[(g',w')]\ra=\sum\limits_{\delta\in\Gamma}\sum\limits_{\chi\in\Gamma^\vee}\la\xi,\delta\cdot\xi'\ra[g^{-1}\delta g']\la f[w],[\chi]e^{2\pi i\la\eta,\delta\ra}f'[w']\ra
\]
where the last inner product in the formula is taken in the algebra $C_0(\ft^*)\rtimes (\Gamma^\vee\rtimes W)$ viewed as a module over itself. The representation of $\C$ is once again given by the trivial projection in $\C[W]$ where $W$ acts diagonally on $H_c, C_c(\ft^*),\Gamma\rtimes W$ and $W$ itself.
The operator is given by $T$ on $H_c$ and by the identity on the other factors. This is a well defined adjointable operator as we took $T$ to be exactly invariant under the action of $\Gamma\rtimes W$ and of finite propagation.

Applying the element $\widehat\beta_{\C}$ reduces this to a module over $C_0(\ft^*)\rtimes (\Gamma^\vee\rtimes W)$ where the inner product is
\[
\la\xi\otimes f[(g,w)],\xi'\otimes f'[(g',w')]\ra=\sum\limits_{\delta\in\Gamma}\sum\limits_{\chi\in\Gamma^\vee}\la\xi,\delta\cdot\xi'\ra\la f[w],[\chi]e^{2\pi i\la\eta,\delta\ra}f'[w']\ra.
\]
Note that as this no longer depends on $g$ and $g'$, vectors of the form $\xi\otimes f[(g_1,w)]$ and $\xi\otimes f[(g_2,w)]$ are identified. Thus the module is really a completion of  $H_c\tensor C_c(\ft^*)\tensor\C[W]$, which we will denote $\cE_1$. Once again the representation is provided by the trivial representation of $W$, and we denote the corresponding projection on $\cE_1$ by $p_W$. The operator on $\cE_1$ is given by $T\otimes 1\otimes 1$.

We now trace the other route around the diagram. As before, starting with a Kasparov triple $(H,\rho, T)$ we obtain the descended triple $(\cE, \hat\rho, T\otimes 1)$. We next apply the element $\alphaCt$ which is given by the completion of $C_c(\ft)$ described earlier in this section. We obtain the completion of $H_c\tensor \C[\Gamma\rtimes W]$ with respect to the inner product 
\[
\la\xi[g],\xi'[g']\ra=\sum\limits_{h\in \Gamma\rtimes W}\la \xi,h\cdot\xi'\ra[g^{-1}hg'].
\]
The representation of $\C$ is given by the identity while the operator, once again, is given by $T$ on $H_c$ and the identity on the other factor.

The Hilbert module realising the descended Morita equivalence is given by completing the module $C_c(\ft^*)\rtimes W$ with respect to the inner product
\[
\la f[w],f'[w']\ra = \sum\limits_{\chi\in \Gamma^\vee}[w^{-1}] \ol f(\chi\cdot f')[\chi w']
\]
in $C_0(\ft^*)\rtimes W$.

The representation of $C^*(\Gamma\rtimes W)$ on this module is given by the representation of $\Gamma\rtimes W$ where $\left((\gamma w')\cdot f[w]\right)(\eta) = e^{2\pi i\la \eta,\gamma\ra} (w'\cdot f)(\eta)[w'w]$.  Hence applying the Morita equivalence we obtain a Kasparov triple where the module is the completion, which we denote by $\cE_2$, of $H_c\tensor C_0(\ft^*)\tensor\C[W]$ with respect to the inner product
\[
\la\xi\otimes f[w],\xi'\otimes f'[w']\ra=\sum\limits_{\delta\in\Gamma}\sum\limits_{u\in W}\sum\limits_{\chi\in \Gamma^\vee}\la \xi,(\delta u)\cdot \xi'\ra \la f[w],e^{2\pi i\la \eta,\delta\ra}[\chi u]f'[w']\ra,
\]
 the representation of $\C$ is given by the identity and the operator is given by $T$ on $H_c$ and the identity elsewhere.

To identify this triple with the Kasparov element obtained via the first route, we note that the module $\cE_2$ is isomorphic to the range of the projection $p_W$ on $\cE_1$. Indeed
\[
\la p_w(\xi\otimes f[w]), p_w(\xi'\otimes f'[w'])\ra_{\cE_1} =\frac 1{|W|}\la \xi\otimes f[w], \xi'\otimes f'[w']\ra_{\cE_2}.
\]
This completes the proof.

\section{Langlands Duality and $K$-theory}\label{the end section}

In this section we will consider the $K$-theory of the affine and extended affine Weyl groups of a compact connected semisimple Lie group.

As remarked in the introduction an extended affine Weyl group and its Langlands dual $(W_a')^\vee$ need not be isomorphic. For example the extended affine Weyl groups of $\PSU_3$ and its Langlands dual $\SU_3$ are non-isomorphic. However their group $C^*$-algebras have the same $K$-theory, see \cite{NPW}. 

In this section we will show that this is not a coincidence, indeed passing to the Langlands dual always rationally preserves the $K$-theory for the extended affine Weyl groups. In particular, where the extended affine Weyl group of the dual of $G$ agrees with the affine Weyl group of $G$ (as for $\PSU_3$) the $K$-theory for the affine and extended affine Weyl groups of $G$ agrees up to rational isomorphism.
 
\begin{theoremn}
\dualthm
\end{theoremn}
\medskip

\begin{proof}   The proof combines the universal coefficient theorem with our Poincar\'e duality as follows.

We start by writing $W_a'=\Gamma\rtimes W$ and $(W_a')^\vee= \Gamma^\vee\rtimes W$. By the Green Julg theorem  and Fourier-Pontryagin duality:

\begin{equation}
K_*(C^*(W_a')^\vee)\simeq K^W_*(C^*(\Gamma^\vee)\simeq K_*^W(C(T)) = K^*_W(T).
\end{equation}

Applying the universal coefficient theorem, we have the exact sequence
\[
0 \to \Ext^1_{\Z}(K^{*-1}_W(T), \Z) \to K_*^W(T) \to \Hom(K^*_W(T),\Z) \to 0
\]
In particular the torsion-free part of $K_*^W(T)$ agrees with the torsion-free part of $K^*_W(T)$ therefore rationally we have
\begin{align}\label{K4}
K_W^*(T)\simeq K_*^W(T).
\end{align}

As in \ref{PD2} we can identify $K_*^W(T)=K_W^*(C(T))$ with $K^*(C^*(W_a')^\vee)$. The theorem now follows by applying our Poincar\'e duality from Theorem \ref{PD2} to obtain

\begin{equation}
K^*(C^*(W_a')^\vee)\cong K_*(C^*(W_a')).
\end{equation}

\end{proof}

In a subsequent paper, \cite{NPW} we construct the admissible duals for the extended affine Weyl groups of all Lie groups of type $A_n$, exhibiting these spaces as varieties which decompose as a union of spaces indexed by the representations of the Weyl group. Furthermore we show that the rational isomorphism given above is induced by a homotopy equivalence between the varieties which respects  the decomposition. The special case of $\SU(n)$ itself  was considered by Solleveld in \cite{S}.
 
 \bigskip
For the affine Weyl groups we have the following:

\begin{corollaryn}
\affinecor
\end{corollaryn}

\begin{proof}
If $G$ is a compact connected semisimple Lie group of adjoint type then its Langlands dual $G^\vee$ is simply connected so $(W_a')^\vee=W_a^\vee$.

In the case that $G$ is additionally of type $A_n, D_n, E_6, E_7, E_8, F_4, G_2$ the group $G^\vee$ is the universal cover of $G$ and hence $W_a=W_a^\vee$.
\end{proof}


\begin{thebibliography}{99999}
\bibitem{BJ}S. Baaj and P. Julg,  Th\'eorie bivariante de Kasparov et op\'erateurs non born\'es
dans les $C^*$-modules hilbertiens, C.R. Acad. Sci. Paris S\'er. I Math. 296 (1983) 875 -- 878. 
\bibitem{Black} B. Blackadar, \emph{$K$-Theory for operator algebras}  (Cambridge University Press, 2002.)
\bibitem{Bl} J. Block, Duality and Equivalence of Module Categories in Noncommutative Geometry, CRM Proceedings and Lecture Notes Volume 50, 2010, 311-340
\bibitem{B} N. Bourbaki, \emph{Lie groups and Lie algebras} (chapters 4-6, Springer 2002;
chapters 7-9, Springer 2008).
\bibitem{BN} U. Bunke and T. Nikolaus, T-duality via gerby geometry and reductions, Rev. Math.
Phys. 27 (2015), no. 5, 1550013, 46
\bibitem{DvE} C.Daenzer and E. van Erp, T-duality for Langlands dual groups, Advances in Theoretical and Mathematical Physics, 18(6), (2014) 1267--1285.
\bibitem{EEK} S. Echterhoff, H. Emerson, H. J. Kim, $KK$-theoretic duality for proper twisted actions, Math. Ann. (2008) 340:839--873.
\bibitem{E} H. Emerson, Lefschetz numbers for C*-algebras. Canad. Math. Bull. 54 (2011), no. 1, 82--99. 
\bibitem{EM} H. Emerson and R. Meyer, Dualities in equivariant Kasparov theory. New York J. Math. 16 (2010), 245--313. 
\bibitem{K} G. Kasparov, Equivariant KK-theory and the Novikov conjecture Invent. math. 91 (1988) 147--201.
\bibitem{M} R. Meyer, Categorical aspects of bivariant K-theory. K-theory and noncommutative geometry, 1--39, EMS Ser. Congr. Rep., Eur. Math. Soc., Z\"urich, 2008.
\bibitem{NPW} G.A.Niblo, R. Plymen, N.J.Wright, Stratified Langlands duality in the $A_n$ tower, arXiv:1611.05218.
\bibitem{S} \mbox{M.S. Solleveld}, Periodic cyclic homology of affine Hecke algebras, PhD Thesis, 2007, FNWI: Korteweg-de Vries Institute for Mathematics (KdVI), http://hdl.handle.net/11245/1.271180.
\end{thebibliography}
\end{document}